%% file: main.tex
\documentclass[11pt]{article}

\usepackage{amsmath,amssymb,theorem,enumerate,graphicx,multirow,multicol,array,arydshln,todonotes,dsfont}
\usepackage{subcaption} 
\usepackage[affil-it]{authblk}

\usepackage[utf8]{inputenc}
\usepackage{url}
\usepackage{xspace}

\usepackage{hyperref}
\usepackage{subcaption} 
\usepackage{adjustbox}

\usepackage{comment}
\usepackage{tikz}
\usepackage{transparent}
\usetikzlibrary{positioning,shapes,chains,arrows.meta}
\usetikzlibrary[calc]
\tikzset{>={Latex[width=1.5mm,length=2.3mm]}}

\tikzstyle{line}=[draw]
\tikzset{%
    node/.style={circle, draw=black, fill=black, minimum size=1mm, inner sep=2pt},
    node2/.style={circle, draw=black, fill=black, minimum size=0.5mm, inner sep=1pt}
}

\usepackage{xspace}
\usepackage{etoolbox}
\usepackage{algorithm}
\usepackage{algorithmicx}
\usepackage[noend]{algpseudocode}

\newtheorem{thm}{Theorem}[section]
\newtheorem{claim}{Claim}[section]
\newtheorem{coro}{Corollary}[section]
\newtheorem{lem}{Lemma}[section]

\newtheorem{prop}{Proposition}[section]

\newtheorem{pf claim}{Proof of Claim}
\newenvironment{proof}[1][Proof]{\noindent \textbf{#1. }}{\ \rule{0.5em}{0.5em}}

\newcommand{\MFS}{{\sc Firebreak Location}\xspace}

\newcommand{\MSFS}{{\sc Windy Firebreak Location}\xspace}

\newcommand{\RSP}{{\sc Restricted Strong Planar Max 2SAT}\xspace}
\newcommand{\GPP}{{\sc Graph Partition}\xspace}

\usepackage{wasysym}
\newcommand\gab[1]{\textcolor{black}{#1}}
\newcommand\pier[1]{\textcolor{black}{#1}}
\newcommand\ales[1]{\textcolor{black}{#1}}
\newcommand\mar[1]{\textcolor{black}{#1}}
\newcommand\newgab[1]{\textcolor{black}{#1}}
\newcommand\newpier[1]{\textcolor{black}{#1}}
\newcommand\newales[1]{\textcolor{black}{#1}}
\newcommand\newmar[1]{\textcolor{black}{#1}}
\newcommand\newmarc[1]{\textcolor{black}{#1}}

\usepackage{newtxtext}

\reversemarginpar

\begin{document}

\title{A graph theoretical approach to the firebreak locating problem\thanks{This work has been supported in part by 
the European project  ``Geospatial based Environment for Optimisation Systems Addressing Fire 
Emergencies'' (GEO-SAFE), contract no. H2020-691161.
}}

\author[1]{Marc Demange}

\affil[1]{\footnotesize School of  Science, RMIT University,
Melbourne, Australia}

\author[2] {Alessia {Di Fonso}}

\affil[2]{\footnotesize Department of Information Engineering, Computer Science and Mathematics, University of L'Aquila, Italy}

\author[2] {Gabriele {Di Stefano}}

\author[3]{Pierpaolo Vittorini}
\affil[3]{\footnotesize Department of Life, Health and Environmental Sciences, University of L'Aquila, Italy}
\affil[ ]{\footnotesize{\tt{marc.demange@rmit.edu.au}, \tt alessia.difonso@graduate.univaq.it, \tt gabriele.distefano@univaq.it, pierpaolo.vittorini@univaq.it }}
\maketitle

\begin{abstract}
\noindent
In the last decade, wildfires have become wider and more destructive. The climate change and the growth of urban areas may further increase the probability of incidence of large-scale fires. 
The risk of fire can be lowered with preventive measures. Among them, firefighting lines are used to stop the fire from spreading beyond them. Due to high costs of installation and maintenance, their placement must be carefully planned. In this work, we address the wildfire management problem from a theoretical point of view and define a risk function to model the fire diffusion phenomena. The land is modeled by a mixed graph in which vertices are areas subject to fire with a certain probability while edges model the probability of fire spreading from one area to another. To reduce the risk, we introduce the \MFS problem that addresses the optimal positioning of firefighting lines under budget constraints. We study the complexity of the problem and prove its hardness even when the graph is planar, bipartite, with maximum degree four and the propagation probabilities are equal to one. We also show an efficient polynomial time algorithm for particular instances on trees.

\end{abstract}


\section{Introduction}

\subsection{Motivation}

Wildfire incidence around the world has increased in the past several decades. The recent fires registered in Chile (2017), Portugal (2017), Greece (2018), California (2017, 2018, 2020), Oregon (2020), Australia (2009, 2019-2020), Sweden (2014, 2018) and UK (2018) are examples of that. The possible causes include the decline of traditional rural systems, the reduction of resources finalized to forest management \cite{Whittaker2012VulnerabilityVictoria}, the increased number of citizens living in wildland-urban interfaces and unaware of the risk, the lack of proper wildfire protection measures \cite{Paveglio2015CategorizingArchetypes} and the extreme weather events caused by global climate change \cite{Lung2013AChange}. These causes have led to larger wildfires, consequent conspicuous damages, that in turn caused broader social, economic and environmental impacts \cite{Turco2018SkilfulPredictions}.

Technologies such as mathematical models predicting fire spread, resource allocation tools or infrastructure allocation decision support systems can be very helpful to support efficient and accurate decision-making \cite{Martell2011TheChallenges}. In such a context, the paper focuses on problems related to the construction of fire-fighting lines, which is one of the most common preventive measures adopted to either isolate valuable areas or to provide opportunities \mar{for} firefighters to face a wildfire \cite{Blasi2004IncendiAmbientale.,PrefetdeCorse2013PlanNaturelsWithURL}. In particular, a {\em firebreak} is a fire-fighting line in which all the vegetation and organic matter are completely removed, that prevents a fire from feeding further and expand towards populated areas and important buildings. Firebreaks require annual maintenance: keeping the optimal functionality is very onerous since it requires maintaining the vegetation level as low as possible while it is known that they create favorable conditions for the growth of invasive vegetation \cite{Blasi2004IncendiAmbientale.,PrefetdeCorse2013PlanNaturelsWithURL}. Because of such an economic, ecological and landscaped impact, it is impossible to create fire-fighting lines everywhere. So far, the installation of firebreaks remains empirical and based on the statistical analysis of past fires and on expert opinions.

\subsection{Our contribution}

In such a context, the paper focuses on some problems related to the optimal  
location of firebreaks in a territory, 
to minimize a risk function under a budget constraint. 

The territory is modeled as a {\em mixed graph} \mar{where vertices correspond to areas subject to some fire outbreak and edges represent the possibility of fire spread from one area to another.} The risk relies on ignition probabilities (that capture the possibility of fire outbreak from each area) and on spread probabilities (that capture the possibility of fire spread from one area to another), as well as on values assigned to each area. \newmar{Directed edges represent situations where the fire can spread in one direction only or in both directions with different probabilities. This can result from the topography (in particular the slope) or/and the dominant wind directions during the year.} 

\newmar{A firebreak corresponds to reducing the graph connectivity by edge removal. A cost is associated with each edge removal, and a budget constraint must be satisfied.
Then, the \MFS problem aims to identifying} the best positions for firebreaks \newmar{to minimize the risk while respecting the budget}. 

\pier{The 
risk function and the firebreak location problem are related to similar approaches 
in the literature. For a better comparison among them, we firstly introduce our model. Then, we delve into the related \mar{works} by comparing the risk function with the well-known {\em Independent Cascade model} \cite{Shakarian2015TheModels} and the firebreak location problem with a selection of problems involving the removal of edges in a graph.}
 
\pier{We then 
\mar{study} the complexity of the problem showing, with two successive reductions, its hardness even when the graph is planar, bipartite, with maximum degree~4, and when vertex values, edge costs and the propagation probabilities are one. 
\mar{We} also 
\mar{propose an efficient polynomial time algorithm for} several cases on trees.
}
\newales{Table \ref{table1} summarizes all the results presented in this paper. The meaning of the columns depend on the model that we introduce in the following section.}

The paper is organized as follows. Section \ref{sec:defmodel} introduces the main notations (Subsection \ref{ssec:notations}), the \newpier{general} problem (Subsection \ref{ssec:problem}) and summarizes the related work (Subsection \ref{ssec:related}). Section \ref{sec:complexity} contains the main results concerning the \newales{computational} complexity, whereas Section \ref{sec:polynomially} studies some polynomially solvable cases on trees. A conclusive discussion is then included in Section \ref{sec:conclusions}.

\begin{table} 
\small
\centering
\begin{tabular}{|c|c|c|c|c|c|c|c|c|c|} \hline
\multicolumn{10}{|c|}{\MFS} \\ \hline
\it Graph $(V,E)$ & $\Delta$ & $\pi_{s}$ & $\pi_{i}$ & $\kappa$ & $\varphi$ & $\mathcal{\mathrm{B}}$ & $R$ & \it Complexity & \it Reference\\ \hline
general&  &  &  &  &  &  &  & NP-hard & Prop. \ref{pro:partition}\\ \hline \hline
\hline \multicolumn{10}{|c|}{\MSFS} \\ \hline
\it Graph $(V,E)$  & $\Delta$ & $\pi_{s}$ & $\pi_{i}$ & $\kappa$ & $\varphi$ & $\mathcal{\mathrm{B}}$ & $R$ & \it Complexity & \it Reference\\ \hline
unit-disk &  &1  & uniform  &  & 1 &  &  & NP-C & \cite{safetyscience2021} \\ \hline
star &  &1  & $0/1$  &  &  &  &  & NP-C & Prop. \ref{pro:partition} \\ \hline
bipartite planar& 4 & 1 &  & 1 & 1 &  &  & NP-C & Th. \ref{theo:bipartite-planar4}\\ \hline
subgrid &4 & 1 &  & 1 & $0/1$ &  &  & NP-C &  Prop. \ref{pro:grid} \\ \hline
grid & 4 & 1 &  & $0/1$ & $0/1$ &  &  & NP-C&  Prop. \ref{pro:grid} \\ \hline
trees (Algo.\ref{treecut})&  & 1 & $0/1$ & 1 & 1 & poly  & opt & $O(|V| \cdot B^2)$ & Th. \ref{Theo-tree}  \\ \hline
trees &  & 1 & $ 0/1 $ & $\mathbb{N}$ &  & poly & opt & $O(|V| \cdot B^2)$ & Th. \ref{Theo-tree2}\\ \hline
\end{tabular}

\caption{Results presented in this paper --- Note: {\it poly} stands for $O(Poly(|V|))$, {\it opt} states that we find the minimum risk value, an empty column means that any value is acceptable}
\label{table1}
\end{table}

\section{Definitions and Model}
\label{sec:defmodel}

\subsection{Main notations}
\label{ssec:notations}

\newmar{In this paper, all graphs are finite.}
Let $G=(V,E)$ be a directed graph (or digraph);  a directed edge $e \in E$ between vertices $x,y \in V$ will be denoted $e=xy$.  


A symmetric digraph is also called undirected graph where two edges $xy,yx$ are replaced with a single undirected edge. When no ambiguity will occur, this undirected edge will be denoted $xy$ (or $yx$). A mixed graph has directed and undirected edges. As a consequence, all graph notions not specifically restricted to directed or to undirected graphs will be defined for mixed graphs and then are valid for directed and for undirected graphs as well. \newmar{In a mixed graph, a path needs to respect edge orientation if it includes directed edges}.

\mar{Given a 
mixed graph $G=(V,E)$ and two vertices $u,v \in V$, $u\stackrel{G}{\rightarrow} v$ means that there is a 
path from $u$ to $v$ in $G$ (we say that $v$ is {\em reachable} from $u$ in $G$).   For any set of vertices $V'\subset V$, $r_{G}(V')=\{v\in V,\exists u\in V', u\stackrel{G}{\rightarrow} v\}$ is the set of vertices reachable from $V'$ in $G$.}

By $G'\leq G$ we will denote that $G'=(V,E'), E'\subset E$ is a  partial graph of $G$. For any edge set $H\subset E$, we denote by $G\setminus H$ the partial graph $(V, E\setminus H)$ obtained from $G$ by removing edges in $H$. Given a set $V'\subset V$, $G[V']$ denotes the subgraph induced by $V'$ and any graph $G''=(V'',E'')$, $V''\subset V, E''\subset E$ will be called partial subgraph of $G$.

For any $k\geq 2$,  $P_k= a_1\ldots a_k$ is a path on $k$ vertices $a_1,\ldots, a_k$ and edges $a_ia_{i+1}$, $i=1, \ldots k-1$. $P_2$ is a single edge. The extremities of the path are the two vertices of degree~1.  Paths can be either undirected or directed from $a_1$ to $a_k$. A cycle $C_k$ is obtained from a $P_k$ by adding an edge between its extremities; if it is directed, then the added edge is $a_ka_1$. 

\mar{A mixed graph} is called planar if it can be drawn in the 2-dimension plane without crossing edges. Such a drawing is called {\em planar embedding}.
Given an undirected graph $G$, a subdivision of $G$ is obtained by replacing edges $uv$ by a path $P_k$, $k\geq 2$ (we add $k-2$ intermediate vertices). This transformation preserves planarity.  
$K_p$, $p\geq 1$ will denote an undirected clique on $p$ vertices and $K_{p,q}$, $p,q\geq 1$, will denote a complete bipartite graph with respectively $p$ and $q$ vertices on each side. It is well-known that $K_{3,3}$ is not planar and consequently,  containing (as partial subgraph) a subdivision of $K_{3,3}$ is a certificate of non-planarity.

All graph-theoretical terms not defined here can be found in~\cite{Diestel2018GraphTheory}. For complexity concepts we refer the reader to~\cite{Garey1979ComputersNP-completeness}.

\subsection{The \MFS Problem}
\label{ssec:problem}

The problem \MFS \newmar{was introduced in~\cite{safetyscience2021} and discussed with firefighting agencies as part of the European project GEO-SAFE. To make this paper self contained, we give below the full description of \MFS and define the particular case \MSFS.}

We are given a mixed
graph where vertices are subject to burn and edges represent potential fire spread from one vertex to an adjacent one. Fire ignitions may occur on some vertices. The objective will be to select a set of 
\newmar{edges}, called {\em cut system}, to be blocked (removed) within a budget constraint in order to reduce the induced risk, as described below.

When a link \newmar{between two vertices} is ``treated'' (blocked) (typically installing a fire break corridor) the fire is blocked in both directions. For this reason a treatment between vertices $x$ and $y$ is modeled as the removal of edges $xy$ and $yx$, if existing. To this end, we define a subset $H$ of $E$ a {\em cut system} if:

\[H = H \cup \left\lbrace yx \in E\ | \ xy \in H \right\rbrace\]

An instance is defined by a mixed 
graph $G=(V,E)$; every 
edge \mar{$e=xy\in E$ is assigned  a probability of spread \mar{(in one direction if the edge is directed)} $\pi_s(e)$ and a cost $\kappa(e)$ seen as the cost to cut $xy$ and $yx$ in case both edges exist. In this case, it will be convenient in expressions to allocate half of the cost to $xy$ and half to $yx$, making $\kappa$ just a symmetric function}. 


Every vertex $v\in V$ \mar{is assigned} a value $\varphi(v)$ and a probability of ignition $\pi_i(v)$. The total value $\varphi(V')$ of a subset $V'\subset V$ is defined in an additive way: 

\begin{equation}
	\varphi(V')=\sum\limits_{v\in V'}\varphi(v)
\end{equation}

The cost of $E' \subset E$ is:

\begin{equation}
	\kappa(E')=\sum\limits_{e \in E'}\kappa(e)
\end{equation} 

The global objective will be to select a {\em cut system} $H \subset E$ 
 under a budget constraint: given a budget $B$, \mar{the constraint is} $\kappa(H)\leq B$.

Given such a solution $H$, we will denote by $G_H$ the partial graph $G_H = G \setminus H$ of $G$.

Given a partial graph $G_S\leq G$ (spreading graph) and a set $I\subset V$ of ignited vertices, the induced loss is:
\begin{equation}
	\lambda(G_S,I)=\varphi(r_{G_S}(I))
\end{equation}

To evaluate the {\em risk} associated with $H$ we consider the following two-phase system.

During a first phase we randomly select a set $I \subset V$ of vertices in the probability space ${\cal V}=(V,\pi_i)$ and a set $E_S\subset (E\setminus H)$ of edges (spreading edges) in the probability space ${\cal E}_H=(E \setminus H,\pi_s)$, thus defining a random graph $G_S=(V,E_S) \leq G_H$. Both random choices are independent. 

The probability of a set $I\subset V$ is:

\begin{equation}\label{eq:probai}
	\pi_i(I)=\prod_{v\in I}\pi_i(v)\times \prod_{v\notin I}(1-\pi_i(v))
\end{equation}

Similarly the probability of $G_S$ is:

\begin{equation}\label{eq:probas}
	\pi_s(G_S)=\prod_{e\in E_S}\pi_s(e)\times \prod_{e\notin E_S}(1-\pi_s(e))
\end{equation}

The second phase then associates with the {\em cut system} $H$ the {\em risk} $\rho(G_H)$ as follows:

\begin{equation}\label{eq:defrisk}  
	\rho(G_H)=\sum_{I\subset V}\sum_{G_S \leq G_H}\pi_i(I)\pi_s(G_S)\lambda(G_S,I)
\end{equation}

\begin{quote}
\MFS\\
{\underline{Instance}}: a mixed 
graph $G=(V,E)$; for every 
edge $e\in E$ a probability of spread $\pi_s(e)$ and a cost $\kappa(e)$ \mar{(if $xy$ and $yx$ are in $E$, then $\kappa(xy)=\kappa(yx)$}; for every vertex $v\in V$, a value $\varphi(v)$ and a probability of ignition $\pi_i(v)$. A total budget $B$ and a total risk $R$. \\
{\underline{Question}}: is there a {\em cut system} $H\subset E$ such that $\kappa(H)\leq B$ and $\rho(G_H)\leq R$?\\
We will denote such an instance $(G,\pi_s,\pi_i,\kappa,\varphi,B,R)$.
\end{quote} 

\mar{If we consider the optimization version instead of the decision version, then the threshold $R$ is not part of the instance but becomes the objective to minimize.}
The particular case where all probabilities of spread are equal to~1 is called \MSFS. In this case, the definition of the problem can be simplified: we can directly define it on a mixed graph, stating that, two opposite edges $xy$ and $yx$ are always seen as an undirected edge. Then,  the cost $\kappa(e)$ to remove an undirected edge $e=xy$ corresponds to the cost to remove $xy$ and $yx$. 

The following lemma  simplifies Relation~\ref{eq:defrisk} in the case of \MSFS. Let us consider an instance of \MSFS:  a digraph $G=(V,E)$  with ignition probabilities on vertices and probability of spread equal to~1 for each edge. For a vertex $x\in V$ we denote $p_x$ the probability that $x$ burns. For any cut system $H$, we denote $U_{x,H}=\{t\in V, t\stackrel{G_H}{\rightarrow} x\}$ be the set of vertices $t$  such that there is, in $G_H$, a path from $t$ to $x$. \mar{Note that $x\in U_{x,H}$. }

\begin{lem}\label{lem:px}
Let us consider an instance $G$ of \MSFS,  a cut system $H$ and  a vertex $x$; then: 
\gab{$$p_x= 1- \prod\limits_{t\in U_{x,H}}(1-\pi_i(t))\  {\rm and}\ 
\rho(G_H)= \sum\limits_{x\in G_H} p_x\cdot \varphi(x).$$}
\end{lem}

\begin{proof}

Using Relation~\ref{eq:probas} we have  
	$\pi_s(G_H)=1$ and $\forall G_S\leq G_H, G_S\neq G_H, \pi_s(G_S)=0$. The expression of $p_x$ is then a direct application of Relations~\ref{eq:probai} and \ref{eq:probas} taking into account that the fire certainly spreads through connected components. Then, the expression of $\rho(G_H)$ is an immediate consequence of Relation~\ref{eq:defrisk} taking into account that, in the second sum,  only the terms with $G_S=G_H$ are not null. 

\end{proof}

Note that, in general, $\rho(G_H)$ is not polynomial to compute as outlined in the Subsection~\ref{ssec:related} and, as a consequence, \MFS is not necessarily in NP.  Lemma~\ref{lem:px} implies that the objective value of \MSFS can be computed in polynomial time. 
Let us remark that $\rho(G )$ is calculable in polynomial time when
$k \log \Delta \in O(\log \log n)$, where $k$ is the maximum  length of a path
between two vertices and $\Delta$ is the maximum vertex degree in the graph.


Under these conditions, the number of nodes that can affect the probability that a node can burn is $O(\Delta^k)$. The number of edges involved is therefore $O(\Delta^{2k})$ and a brute-force algorithm has to test $O(2^{\Delta^{2k}})$ realizations of the fire propagation on the subgraph, for each burning node of the graph. Imposing a polynomial time for this operation, we obtain $2^{\Delta^{2k}}\in O(n^t)$, for a constant $t$, and hence -- applying two times the logarithmic operator -- we have $k \log \Delta \in O(\log \log n)$.

\subsection{Related Work}
\label{ssec:related}


\subsubsection{The risk function and the Independent Cascade model}

The {\em risk} $\rho(G_H)$ is related to the well-known Independent Cascade model (IC)~\cite{Shakarian2015TheModels}.

The IC model was initially adopted to model the interaction between particle systems \cite{Durrett1988}, then investigated in the context of marketing \cite{Goldenberg2001}, and also used to model the spread of influence in social networks \cite{Kempe2015MaximizingNetwork} or the spread of a disease~\cite{Vittorini2010AnDiseases}. 

The IC model corresponds to the propagation process we have modeled for the case of fire. It is commonly adopted in various contexts where something (let us call it the {\em object}) is spread across a network. Affected vertices are called {\em activated}, which in our case corresponds to be on fire. The model uses discrete time: neighbors of activated vertices are randomly activated using transmission probabilities on links. 
The only difference with our model is the notion of ignition probabilities while in the usual IC model one single vertex or a set of vertices are initially activated. Nevertheless, we could easily remove ignition probabilities   by (i) introducing in our model a vertex $f$ (i.e., the fire) with $\varphi(f)=0$, (ii)  adding for each vertex $v \in V$ an edge $e=fv$ with $\pi_s(e)=\pi_i(v)$ and $\kappa(e)=\infty$, and then (iii) removing the ignition probabilities. However, such a transformation does not necessarily preserve structural properties of the network, in particular planarity that plays an important role in our study.

Computing the risk $\rho(G_H)$ was shown \#P-hard in general graphs~\cite{wang2012scalable} and \#P-complete in planar graphs~\cite{Provan1986}, even with binary ignition probabilities. Approximating $\rho(G_H)$ in the unweighted case is discussed in~\cite{Kempe2015MaximizingNetwork}, where an $\varepsilon$-approximation (approximation schema) \newgab{is provided} under some assumptions. To our knowledge, the possibility to compute an approximated value in the general case is still open. Efficient ways to compute the exact value of $\rho(G_H)$ are also discussed in~\cite{Maehara2017}.

In some contexts and in particular for the spread of an information in a social network,  the so-called  Linear Threshold (LT) model~\cite{Shakarian2015TheModels} is preferred to the IC model. Instead of activating a vertex when the object is transmitted from a neighbor, a given vertex needs to receive a fixed amount of activation from the neighborhood to be itself activated.

\subsubsection{Close optimization problems in other contexts. }
Graph optimization has been used to control the spread of an unwanted object in a network and an extended literature has been developed on this problem. Most models address the spread of false information in social networks or the spread of a disease in the graph of contacts associated with a population.

In these contexts, let us outline two problems, the \textsc{CuttingEdge} problem~\cite{khalil2013cuttingedge} and the \textsc{Contamination Minimization} problem~\cite{Kimura2009}. Both problems have very similar objectives as \MFS and consists in removing a fixed number of edges to minimize the risk, seen as the expected number of activated vertices.  The former problem uses the LT model while the \newgab{latter} uses the IC model. 



A greedy approach is known to guarantee a constant approximation ratio for the \textsc{CuttingEdge} problem~\cite{khalil2013cuttingedge}.
In~\cite{Kimura2009}, some greedy heuristics are considered for \textsc{Contamination Minimization} problem and, to our knowledge, no constant approximation is known. 

Apart from the diffusion model \-- LT or IC \--, the main differences with \MFS are as follows:
\begin{enumerate}
    \item these two problems are unweighted in the sense that all activation probabilities are supposed equal, there is no value associated with vertices and no cost associated with edges.
    \item the \newgab{most important} is the edge removal process: both problems consider removing directed edges while, in \MFS, removing one directed edge automatically removes the inverse edge if it exists.
\end{enumerate}


Finally, a slightly different approach is proposed in~\cite{Tong}. Based on the assumption that the risk of spread is linked to the largest eigenvalue of the adjacency graph, they consider the problem of removing a fixed number of edges to minimize the maximum eigenvalue in the resulting graph. They prove the NP-hardness and propose an heuristic for this problem.

Note that  in most studies aiming to control the spread of information through a social network or the spread of a disease across a population, the underlying graph is dense and the natural topologies of the network are totally different than for our problem. In particular, planarity is not relevant for these networks while it is very natural in our case and is central in our work.

\subsubsection{Close problems in fire emergency management}

We conclude this section by mentioning a few problems in fire emergency management that are naturally posed on the fire spread graph and consequently, for which planar instances are relevant.

The same ignition and propagation model as ours has been used on a directed \newmar{graph} to model and simulate fire spread in a building~\cite{bayesian-building} or on a landscape~\cite{Begonia-Adan2020}. These papers  refer to the Bayesian networks terminology \newmar{with the restriction that the graph of spread has no directed cycle (directed acyclic graph).}   A similar approach is used in~\cite{Mahmoud2018UnravelingFires}. However, in these papers, the model allows to simulate the spread and then, is used as a decision tool to evaluate scenarios. Here,  we use it to  analytically define the objective value of our minimization problem. \newmar{Also, in practice, the case of directed acyclic graphs of spread is satisfactory when considering a specific fire with, in particular,  specific wind directions, but becomes restrictive when considering all possible fire scenarios. The problem we address, in prevention, requires considering all possible scenarios during a long period (typically several years). }

In \cite{Russo2016AFires}, the terrain is modeled as a lattice  (obtained by tessellation), where each node is susceptible to burn, and the fire propagation is represented as a random walk from a starting vertex on fire. Centrality statistics on the adjacency matrix is used for the identification of nodes that contribute the most to the fire spread through the network. These  vertices are marked as the places where to build fuel breaks. 

The difference with our model is twofold: first, firebreaks are meant to be placed on vertices and not on edges; this changes the combinatorial structure of the problem. Their model could be transformed into our setup transforming every node $x$ with two nodes $x^+, x^-$ linked by one  directed edge from $x^+$ to $x^-$ in such a way  that all other edges adjacent to $x^+$ (resp., $x^-$) are entering (resp., exiting) edges. Cost system could impose that only these edges can be part of a cut. On the other hand, in the undirected case, our model could be expressed as a node-based model in the line graph with node set the edges of the original graph and edges between two nodes representing adjacent edges. However, in this case, the spread of fire in the original network cannot be easily represented with the spread in the obtained network. Secondly, the objective function in~\cite{Russo2016AFires} is not directly linked to the risk we use in this work.

A few other combinatorial problems  have been studied on a similar fire-spread graph in a wild fire management context. In particular, the \textsc{Fuel Management} problem aims to schedule fuel reduction in the landscape in order to reduce the risk of spread.  The instance is similar as for \MFS. The main difference is that treating a zone (prescribed fire or harvesting) corresponds to removing all edges adjacent to the related vertex, which, again, makes this problem vertex-based. A second difference is that a vertex removed from the graph may reappear after a few years with the natural growth of  the vegetation. The objective is to schedule preventive treatments over a long period.  To our knowledge, most approaches to solve 
\textsc{Fuel Management} use integer linear formulations for which it is hard to exploit planarity. In~\cite{icores16} however, a multi-period \textsc{Fuel Management} problem is considered from a graph optimization perspective. The problem is reduced to a \textsc{Vertex Cover} problem and is shown hard on planar graphs. Planarity is then exploited to derive an asymptotic approximation  schema.

Finally, we have introduced the problem \MFS in~\cite{safetyscience2021}. We have proved in particular that \MSFS is hard in unit-disk graphs (intersection graph of disks of fixed radius), as explained in Section~\ref{sec:complexity}. This class of graphs is natural when considering embers-based fire spread. We have proposed as well a heuristic for this problem and some generalizations.  We have implemented and tested it as part of web-application developed in collaboration with a fire agency in Corsica (France). First results are very encouraging.

\section{Complexity results for planar instances}
\label{sec:complexity}

\newmar{In~\cite{safetyscience2021}, we noticed that, if all ignition probabilities are equal, then \MSFS in an undirected graph can be reduced to the $k$-\GPP problem that consists in removing a fixed number of edges to split a graph in $k$ connected components of balanced size. This argument allows to prove the hardness of \MSFS in any class of graph where   $k$-\GPP is hard and in particular~\cite{Diaz-UDG} in unit-disk graphs. To our knowledge, the complexity of $k$-\GPP is not known in restricted classes of planar graphs that are the most natural case for our problem in the context of wildfire management.} 

\subsection{A restricted version of {\sc Planar Max 2SAT}}\label{subsec:restricted planar 2-SAT}
\newmar{To establish Proposition~\ref{prop:complex-Max 2SAT}, one of the main results of this section,} 
we will need a restricted version of {\sc Max 2SAT}, the Maximum 2-Satisfiability problem (see~\cite{Garey1979ComputersNP-completeness}).  

A {\sc SAT} instance $\Phi$ is defined as a set $X$ of $n$ boolean variables and a set $C$ of $m$ clauses, each  defined as a set of literals: every variable $x$ corresponds to two literals $x$ (the positive form) and $\bar x$ (the negative form). A $k$-clause, $k\geq 1$, is a clause containing exactly $k$ literals, all different. To simplify the notations we will denote by $(\ell_1, \ldots, \ell_k)$ the $k$-clause with literals $\ell_i, i=1, \ldots k$, without distinction between the orders in which they are listed. A truth assignment assigns a boolean value (True or False) to each variable corresponding to a truth assignment of opposite values for the two literals $x$ and $\bar x$: $\bar x$ is True if and only if $x$ is False. For a literal $\ell\in\{x, \bar x\}$ we denote by $\bar \ell$ its negation: $\bar \ell=\bar x$ if $\ell= x$ and $\bar \ell= x$ if $\ell = \bar x$. A clause is satisfied if at least one of its literals is satisfied. The usual SAT problem asks whether there is a truth assignment satisfying all clauses and 3SAT is the  restriction where $\Phi$ contains only 3-clauses. 
In what follows,  we will assume that no clause contains two opposite literals ($x$ and $\bar x$), in which case the clause would be always satisfied.\\

A {\sc Max 2SAT} instance is a SAT instance $\Phi$ with only 2-clauses but this time the aim is to find a truth assignment on  $X$ maximizing the number of  clauses that are satisfied.  

In planar versions of SAT problems one usually considers the bipartite graph $G_\Phi=(X\cup C, E)$ with an edge $xc\in E$  between a variable vertex $x$ and a clause vertex $c$ if either $x$ or $\bar x$  appears  in the clause $c$.  {\sc Planar Max 2SAT} will be the restricted version of {\sc Max 2SAT} when the graph $G_\Phi$ is planar. Edges of $G_\Phi$ can be labelled with literals with all edges incident to a variable vertex associated with variable $x$ are labelled either with $x$ or with $\bar x$.

Here, we will need a restricted condition of planarity, called {\em strong planarity}, ensuring that there is a planar embedding of $G_\Phi$ such that, for every variable vertex $x$, all the edges incident to $x$  and with the same label can be ``grouped'' on the same side given an orientation of the 2 dimensional plane.  More formally we use the graph $\widetilde{G}_\Phi$ obtained from $G_\Phi$ by replacing variable vertices by a path $P_3$ as follows:
\begin{quote}
  {\bf -} Every variable $x$ is associated with a path \mar{$P_3$}  $xx'\bar x$ with vertices $x,x',\bar x$ and edges $xx', x'\bar x$. \\
 {\bf -} Every clause $c$ is associated with a vertex $c$.\\
{\bf -} There is an edge $cx$ (resp. $c\bar x$) if the literal $x$ (resp. $\bar x$) appears in the clause~$c$.
\end{quote}

A SAT instance $\Phi$ will be called {\em strongly planar} if  $\widetilde{G}_\Phi$ is planar.\\

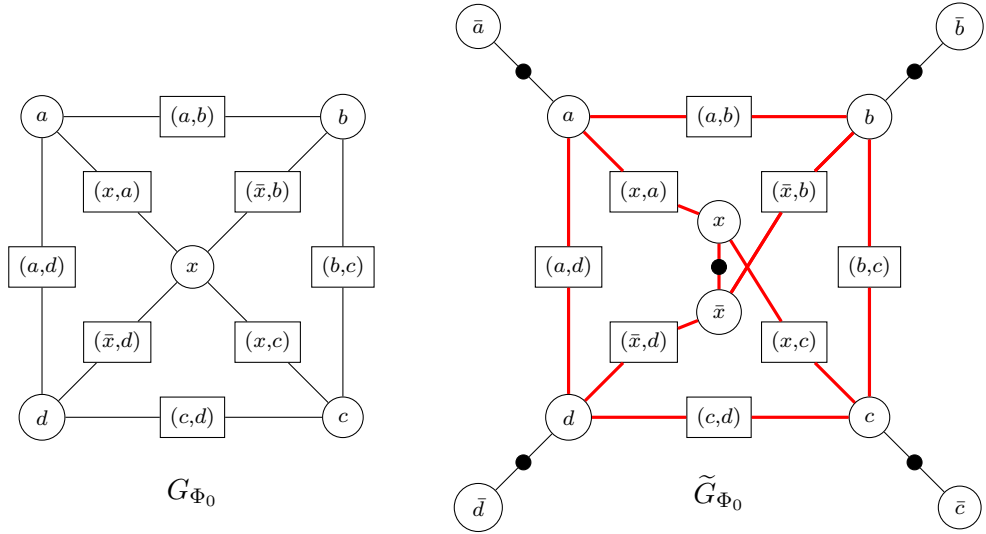
\begin{figure}[tb]
  \begin{center}
    \begin{tikzpicture}[node distance=1cm and 1cm]

\tikzstyle var=[shape=circle,draw]
\tikzstyle clause=[shape=rectangle,draw]

    \node[style = var] (x) at (-5,0) {$\scriptstyle x$};
     \node[style = clause] (xa) at (-6,1) {${\scriptstyle(x,a)}$};
      \node[style = clause] (xc) at (-4,-1) {${\scriptstyle(x,c)}$};
   \node[style = clause](xbb) at (-4,1) {${\scriptstyle(\bar x,b)}$};
   \node[style = clause] (xbd) at (-6,-1){${\scriptstyle(\bar x,d)}$};
   \node[style = var]  (a) at (-7,2) {$ \scriptstyle a$};
\node[style = var]  (c) at (-3,-2) {$\scriptstyle c$};
\node[style = var]  (b) at (-3,2) {$\scriptstyle b$};
  \node[style = var]  (d) at (-7,-2) {$\scriptstyle d$};
   \node[style = clause] (ab) at (-5,2) {${\scriptstyle(a,b)}$};
 \node[style = clause] (bc) at (-3,0) {${\scriptstyle(b,c)}$};
 \node[style = clause] (cd) at (-5,-2) {${\scriptstyle(c,d)}$};
\node[style = clause](ad) at (-7,0) {${\scriptstyle(a,d)}$};

     \node[style = clause] (xa2) at (1,1) {${\scriptstyle(x,a)}$};
      \node[style = clause] (xc2) at (3,-1) {${\scriptstyle(x,c)}$};
   \node[style = clause](xbb2) at (3,1) {${\scriptstyle(\bar x,b)}$};
   \node[style = clause] (xbd2) at (1,-1){${\scriptstyle(\bar x,d)}$};
   \node[style = var]  (a2) at (0,2) {$ \scriptstyle a$};
\node[style = var]  (c2) at (4,-2) {$\scriptstyle c$};
\node[style = var]  (b2) at (4,2) {$\scriptstyle b$};
  \node[style = var]  (d2) at (0,-2) {$\scriptstyle d$};
   \node[style = clause] (ab2) at (2,2) {${\scriptstyle(a,b)}$};
 \node[style = clause] (bc2) at (4,0) {${\scriptstyle(b,c)}$};
 \node[style = clause] (cd2) at (2,-2) {${\scriptstyle(c,d)}$};
\node[style = clause](ad2) at (0,0) {${\scriptstyle(a,d)}$};

 \node[style = var] (x2) at (2,.6) {$\scriptstyle x$};
  \node[style = var] (x4) at (2,-.6) {$\scriptstyle \bar x$};
    \node[node] (x3) at (2,0) {};

 \node[style = var]  (a4) at (-1.2,3.2) {$ \scriptstyle \bar a$};
    \node[node] (a3) at (-0.6,2.6) {};

\node[style = var]  (b4) at (5.2,3.2) {$\scriptstyle \bar b$};
    \node[node] (b3) at (4.6,2.6) {};

\node[style = var]  (c4) at (5.2,-3.2) {$\scriptstyle \bar c$};
    \node[node] (c3) at (4.6,-2.6) {};

 \node[style = var]  (d4) at (-1.2,-3.2) {$\scriptstyle \bar d$};  
     \node[node] (d3) at (-0.6,-2.6) {};
\foreach \from/\to in {a/xa,a/ab,a/ad,x/xa,x/xbb,x/xbd,x/xc,d/ad,d/xbd,d/cd,c/cd,c/xc,c/bc,b/bc,b/xbb,b/ab}
 \draw (\from) edge node{} (\to) ;
 \foreach \from/\to in {a2/xa2,a2/ab2,a2/ad2,x2/xa2,x4/xbb2,x4/xbd2,x2/xc2,d2/ad2,d2/xbd2,d2/cd2,c2/cd2,c2/xc2,c2/bc2,b2/bc2,b2/xbb2,b2/ab2}
 \draw (\from) edge node{} (\to) ;
  \foreach \from/\to in {x2/x3,x3/x4,a2/a3,a3/a4,b2/b3,b3/b4,c2/c3,c3/c4,d2/d3,d3/d4}
 \draw (\from) edge node{} (\to) ;

\foreach \from/\to in {a2/xa2,x2/xa2,x2/xc2, c2/xc2, x2/x3,x3/x4, d2/xbd2, x4/xbd2,d2/ad2, a2/ad2, c2/cd2, d2/cd2, b2/xbb2, x4/xbb2,b2/xbb2, x4/xbb2,b2/ab2, a2/ab2, c2/bc2, b2/bc2}
 \draw[very thick,red] (\from) edge node{} (\to) ;

     \node[rectangle, minimum width=2cm] (gphi0) at (-5,-3) {$G_{\Phi_0}$};
     \node[rectangle, minimum width=2cm] (tildegphi0) at (2,-3) {$\widetilde G_{\Phi_0}$};

      \end{tikzpicture}
  \end{center}
  \caption{An example of 2-SAT instance $\Phi_0=(X_0,C_0)$  that is planar but not strongly planar.  $X_0=\{x,a,b,c,d\}$ and $C_0=\{(x,a), (x,c), (\bar x,b), (\bar x, d),  (a,b), (b,c),(c,d), (a,d)\}$. A planar embedding of $G_{\Phi_0}$ (left) and in $\widetilde G_{\Phi_0}$ (right), the edges of the subdivision of $K_{3,3}$ are highlighted in red. To simplify the figure, vertices $x',a',b',c'$ and $d'$ are represented by a black vertex without label.}
  \label{fig:Phi0}
\end{figure}


Note that $\widetilde{G}_\Phi$ is still bipartite; moreover if   $\widetilde{G}_\Phi$ is planar, then ${G_\Phi}$ is also planar but the converse is not true. Consider for instance the instance $\Phi_0$ with variables $X_0=\{x,a,b,c,d\}$ and clauses $C_0=\{(x,a), (x,c), (\bar x,b), (\bar x, d),  (a,b), (b,c),$ $(c,d), (a,d)\}$ (see Figure~\ref{fig:Phi0}). $G_{\Phi_0}$ is planar since it is a subdivision of the graph obtained from a $C_4$ $abcd$ by adding a universal vertex $x$ (a wheel on five vertices). However, $\widetilde G_{\Phi_0}$ is not planar since it contains a subdivision of the complete bipartite graph $K_{3,3}$ 
which is not planar.

{\sc Strong Planar Max 2SAT} is defined as the restriction of  {\sc Max 2SAT} for which $\widetilde{G}_\Phi$ is planar; it is a restricted case of {\sc Planar Max 2SAT}. We then consider the  \RSP defined as follows (decision version):
\begin{quote}
\RSP\\
{\underline{Instance}}: a {\sc Max 2SAT} instance  $\Phi=(X,C)$ defined by a set of boolean variables $X$ and a set of 2-clauses $C$ as well as an integer $K\leq |C|$ with the following restrictions:\\
\mbox{}\qquad (i): $\widetilde{G}_\Phi$ is planar; \\
\mbox{}\qquad(ii): each literal appears in at most four clauses \mar{(So, $\widetilde{G}_\Phi$ is of maximum degree~5)}.\\
{\underline{Question}}: is there a truth assignment of variables such that at least $K$ clauses are satisfied? \\
We will denote such an instance $(\Phi,K)$ or $(X,C,K)$.
\end{quote}

We are now ready to establish the following proposition  used in the next section. The decision version of {\sc Max 2SAT} is known to be NP-complete~\cite{Garey1976SomeProblems} and moreover, in~\cite{Guibas1993ApproximatingPaths} it is noticed that the reduction preserves planarity, thus inducing that the decision version of {\sc Planar Max 2SAT} is NP-complete using {\sc Planar 3SAT}~\cite{lichtenstein1982planar}. Unfortunately, the reduction given in~\cite{Garey1976SomeProblems} does not preserve the strong planarity property and we will need to slightly modify and rewrite the argument given in~\cite{Guibas1993ApproximatingPaths}.  

\begin{prop}\label{prop:complex-Max 2SAT}
	\RSP is NP-complete.
\end{prop}
\begin{proof}
Given a truth assignment one can decide in $O(m+n)$ whether it satisfies at least $K$ clauses and consequently the problem is in NP. We will now show that a restricted version of {\sc Planar 3SAT}, shown as NP-complete in~\cite{dahlhaus1994complexity}, polynomially reduces to \RSP. In~\cite{dahlhaus1994complexity}, it is proved that the following problem is NP-complete: 

\begin{quote}
	An instance $\Phi$ is defined as a set $X$ of boolean variables and a set $C$ of 2-clauses and 3-clauses. The graph $G_\Phi$ is planar and every variable $x$ appears in exactly two 2-clauses and one 3-clause. Moreover, $x$ appears once in positive form (literal $x$) and once in negative form (literal $\bar x$) in the 2-clauses. 
\end{quote}

 Consider now such an instance $\Phi=(X,C)$. Denote $C=C_2\cup C_3$ where $C_2$ is the set of 2-clauses and $C_3$ the set of 3-clauses.  Since every variable appears  in exactly three clauses, if $G_\Phi$ is planar, so does $\widetilde{G}_\Phi$. We describe below how to build from $\Phi=(X,C)$ an instance $(\Phi^\ast,K)=(X^\ast, C^\ast,K)$ of \RSP such that $\Phi$ is satisfiable if and only if $(\Phi^\ast,K)$ is satisfiable.\\
 
 \newpage
 
\underline{Case of 3-clauses}:
 
For every 3-clause $c=(\ell_1, \ell_2,\ell_3)$, where $\ell_i$, $i=1,2,3$ are literals associated with variables in $X$, we introduce a new variable $a_c$ and replace $c$ with six 2-clauses  and four 1-clauses (1-clauses will be replaced later by two 2-clauses): $c$ is replaced by the set of ten clauses ${\cal C}_c=\{(\ell_1,\ell_2), (\ell_1,\ell_3), (\ell_2,\ell_3), (\ell_1,a_c), (\ell_2,a_c), (\ell_3,a_c), (\bar{\ell_1}), (\bar{\ell_2}), (\bar{\ell_3}), (\bar{ a_c})\}$ (see Figure~\ref{fig:3-clause}). \\
 
We then emphasize the following property:
 
\begin{claim}
 For every $c\in C_3$, at most seven clauses in ${\cal C}_c$ can be simultaneously satisfied. 
\end{claim}\label{claim:max7}
\begin{proof}
Suppose $a_c$ is False, then at most three clauses among $(\ell_1,a_c), (\ell_2,a_c), (\ell_3,a_c)$, $(\bar{\ell_1})$, $(\bar{\ell_2})$ and $(\bar{\ell_3})$ can be True simultaneously, thus at least three clauses are False. If $a_c$ is True and at least two literals among $\ell_1,\ell_2,\ell_3$ are True, then at least two clauses among $(\bar{\ell_1}), (\bar{\ell_2}), (\bar{\ell_3})$ are  False; with $(\bar {a_c})$, it makes at least three unsatisfied clauses. Finally if $a_c$ is True and at least two  literals among $\ell_1,\ell_2,\ell_3$ \-- say without loss of generality $\ell_1,\ell_2$ \-- are False,  then the clause $(\ell_1,\ell_2)$ is False and the three remaining clauses $(\ell_1,\ell_3), (\ell_2,\ell_3), (\bar{\ell_3})$ cannot be simultaneously satisfied, leaving at least three unsatisfied clauses in ${\cal C}_c$.
\end{proof}

 \begin{figure}[hbtp]
  \begin{center}
    \begin{tikzpicture}[node distance=1cm and 1cm]

\tikzstyle var2=[shape=circle]
\tikzstyle clause2=[shape=rectangle]
    \node[style = var2] (pl1) at (0,7.1) {$\scriptscriptstyle \ell_1$};
    \node[style = var2] (pbl1) at (0,5.7) {$\scriptscriptstyle \bar{\ell}_1$};
    \node[node2] (pl1') at (0,6.4) {};
   
     \node[style = clause2] (l1l2l3) at (0,8.4) {$\scriptscriptstyle (\ell_1,\ell_2,\ell_3)$};
      \node[style = clause2] (pbl1l5) at (-2.5,5) {$\scriptscriptstyle (\bar\ell_1,\ell_5)$};
      \node[style = clause2] (pl1l4) at (2.5,5) {$\scriptscriptstyle (\ell_1,\ell_4)$};
     
   \node[style = var2] (pl2) at (-1.8,9.3) {$\scriptscriptstyle \ell_2$};
    \node[style = var2] (pbl2) at (-3.2,9.3) {$\scriptscriptstyle \bar{\ell}_2$};
    \node[node2] (pl2') at (-2.5,9.3) {};

 \node[style = var2] (pl3) at (1.8,9.3) {$\scriptscriptstyle \ell_3$};
    \node[style = var2] (pbl3) at (3.2,9.3) {$\scriptscriptstyle \bar{\ell}_3$};
    \node[node2] (pl3') at (2.5,9.3) {};

 \node[style =var2] (pl5) at (-3.1,4.1) {};
    \node[style =var2] (pl4) at (3.1,4.1) {};

\node[style =var2] (pcl2) at (-2.2,10.2) {};
    \node[style =var2] (pcbl2) at (-3.9,8.4) {};

\node[style =var2] (pcl3) at (2.2,10.2) {};
    \node[style =var2] (pcbl3) at (3.9,8.4) {};
    \foreach \from/\to in {pl1/pl1',pl1'/pbl1,pl2/pl2',pl2'/pbl2,pl3/pl3',pl3'/pbl3}
\draw (\from) edge node{} (\to) ;
 
   \foreach \from/\to in {pl1/l1l2l3,l1l2l3/pl2,l1l2l3/pl2,l1l2l3/pl3,pbl1/pbl1l5,pl1l4/pl1}
\draw (\from) edge node{} (\to) ;

\foreach \from/\to in {pl4/pl1l4,pl5/pbl1l5,pl2/pcl2,pbl2/pcbl2,pl3/pcl3,pbl3/pcbl3}
\draw[dashed] (\from) -- (\to) ;

    \draw [dashed] (-8,4) -- (8,4);
 
  \node[style = var2] (ac) at (0,0) {$\scriptscriptstyle a_c$};
    \node[style = var2] (bac) at (1.4,0) {$\scriptscriptstyle \bar{a}_c$};
    \node[node2] (ac') at (.7,0) {};
     \node[style = clause2] (bacra) at (2.1,.7) {$\scriptscriptstyle (\bar{a}_c,r_c^a)$};
      \node[style = clause2] (bacbra) at (2.1,-.7) {$\scriptscriptstyle (\bar{a}_c,
      \bar{r}_c^a)$};
\node[style =var2] (ra) at (3.2,.7) {$\scriptscriptstyle r_c^a$};
\node[style =var2] (bra) at (3.2,-.7) {$\scriptscriptstyle \bar{r}_c^a$};
  \node[node2] (ra') at (3.2,0) {};
  
    \node[style = var2] (l1) at (0,-2.5) {$\scriptscriptstyle \ell_1$};
    \node[style = var2] (bl1) at (0,-3.9) {$\scriptscriptstyle \bar{\ell}_1$};
    \node[node2] (l1') at (0,-3.2) {};
     \node[style = clause2] (bl1r1) at (.7,-4.6) {$\scriptscriptstyle (\bar{\ell}_1,r_c^1)$};
      \node[style = clause2] (bl1br1) at (-.7,-4.6) {$\scriptscriptstyle (\bar{\ell}_1,
      \bar{r}_c^1)$};
\node[style =var2] (r1) at (.7,-5.6) {$\scriptscriptstyle r_c^1$};
\node[style =var2] (br1) at (-0.7,-5.6) {$\scriptscriptstyle \bar{r}_c^1$};
  \node[node2] (r1') at (0,-5.6) {};
    
   \node[style = var2] (l2) at (-2.6,2.5) {$\scriptscriptstyle \ell_2$};
    \node[style = var2] (bl2) at (-4,2.5) {$\scriptscriptstyle \bar{\ell}_2$};
    \node[node2] (l2') at (-3.3,2.5) {};
     \node[style = clause2] (bl2r2) at (-4.7,3.2) {$\scriptscriptstyle (\bar{\ell}_2,r_c^2)$};
      \node[style = clause2] (bl2br2) at (-4.7,1.8) {$\scriptscriptstyle (\bar{\ell}_2,
      \bar{r}_c^2)$};
\node[style =var2] (r2) at (-5.8,3.2) {$\scriptscriptstyle r_c^2$};
\node[style =var2] (br2) at (-5.8,1.8) {$\scriptscriptstyle \bar{r}_c^2$};
  \node[node2] (r2') at (-5.8,2.5) {};

 \node[style = var2] (l3) at (2.6,2.5) {$\scriptscriptstyle \ell_3$};
    \node[style = var2] (bl3) at (4,2.5) {$\scriptscriptstyle \bar{\ell}_3$};
    \node[node2] (l3') at (3.3,2.5) {};
     \node[style = clause2] (bl3r3) at (4.7,3.2) {$\scriptscriptstyle (\bar{\ell}_3,r_c^3)$};
      \node[style = clause2] (bl3br3) at (4.7,1.8) {$\scriptscriptstyle (\bar{\ell}_3,
      \bar{r}_c^3)$};
\node[style =var2] (r3) at (5.8,3.2) {$\scriptscriptstyle r_c^3$};
\node[style =var2] (br3) at (5.8,1.8) {$\scriptscriptstyle \bar{r}_c^3$};
  \node[node2] (r3') at (5.8,2.5) {};

 \node[style = clause2] (acl1) at (0,-1.25) {$\scriptscriptstyle (a_c,\ell_1)$};
 \node[style = clause2] (acl2) at (-1.3,1.25) {$\scriptscriptstyle (a_c,\ell_2)$};
 \node[style = clause2] (acl3) at (1.3,1.25) {$\scriptscriptstyle (a_c,\ell_3)$};
 
 \node[style = clause2] (l1l2) at (-5,-1) {$\scriptscriptstyle (\ell_1,\ell_2)$};
  \node[style = clause2] (l2l3) at (0,3.5) {$\scriptscriptstyle (\ell_2,\ell_3)$};
   \node[style = clause2] (l1l3) at (5,-1) {$\scriptscriptstyle (\ell_1,\ell_3)$};
  
    \node[style = clause2] (l1l4) at (4.5,-4.5) {$\scriptscriptstyle (\ell_1,\ell_4)$};
   \node[style = clause2] (bl1l5) at (-4.5,-4.5) {$\scriptscriptstyle (\bar\ell_1,\ell_5)$};
   
   \node[style =var2] (l5) at (-5,-5.6) {};
    \node[style =var2] (l4) at (5,-5.6) {};
    
     \node[style =var2] (cl2) at (-3,3.4) {};
    \node[style =var2] (cbl2) at (-4.7,1) {};
    
    \node[style =var2] (cl3) at (3,3.4) {};
    \node[style =var2] (cbl3) at (4.7,1) {};
    \foreach \from/\to in {ac/ac',ac'/bac,bac/bacra, bac/bacbra,bacra/ra,ra/ra',ra'/bra,bacbra/bra}
\draw (\from) edge node{} (\to) ;
     \foreach \from/\to in {l1/l1',l1'/bl1,bl1/bl1r1, bl1/bl1br1,bl1r1/r1,r1/r1',r1'/br1,bl1br1/br1}
\draw (\from) edge node{} (\to) ;
   \foreach \from/\to in {l2/l2',l2'/bl2,bl2/bl2r2, bl2/bl2br2,bl2r2/r2,r2/r2',r2'/br2,bl2br2/br2}
\draw (\from) edge node{} (\to) ;
 \foreach \from/\to in {l3/l3',l3'/bl3,bl3/bl3r3, bl3/bl3br3,bl3r3/r3,r3/r3',r3'/br3,bl3br3/br3}
\draw (\from) edge node{} (\to) ;

\foreach \from/\to in {ac/acl1,acl1/l1,ac/acl2,acl2/l2,ac/acl3,acl3/l3}
\draw (\from) edge node{} (\to) ;

\foreach \from/\to in {l1/l1l2,l1l2/l2,l2/l2l3,l2l3/l3,l1/l1l3,l1l3/l3}
\draw (\from) edge node{} (\to) ;

\foreach \from/\to in {l1/l1l4,bl1/bl1l5}
\draw (\from) edge node{} (\to) ;

\foreach \from/\to in {l4/l1l4,l5/bl1l5,l2/cl2,l3/cl3}
\draw[dashed] (\from) -- (\to) ;

\draw[dashed] (bl2) to [out=-90,in=30] (cbl2);
\draw[dashed] (bl3) to [out=-90,in=-30] (cbl3);
 \node[rectangle, minimum width=2cm] (wgphi) at (-7.1,7) { $\widetilde G_\Phi$};
  \node[rectangle, minimum width=2cm] (wgphi*) at (-7,-1) { $\widetilde {G}_{\Phi^\ast}$};

      \end{tikzpicture}
  \end{center}
  \caption{\mar{Gadget associated with the 3-clause $c=(\ell_1, \ell_2, \ell_3)$ in $\widetilde G_\Phi$ (up) and $\widetilde G_{\Phi^\ast}$ (bottom).
  The literal  $\ell_1$ appears in one 3-clause $(\ell_1,\ell_2,\ell_3)$ and one 2-clause $(\ell_1,\ell_4)$ while $\bar\ell_1$ appears in one 2-clause  $(\bar\ell_1,\ell_5)$. Black vertices correspond to $x'$-vertices for variables in $X^\ast$.}}
  \label{fig:3-clause}
\end{figure}
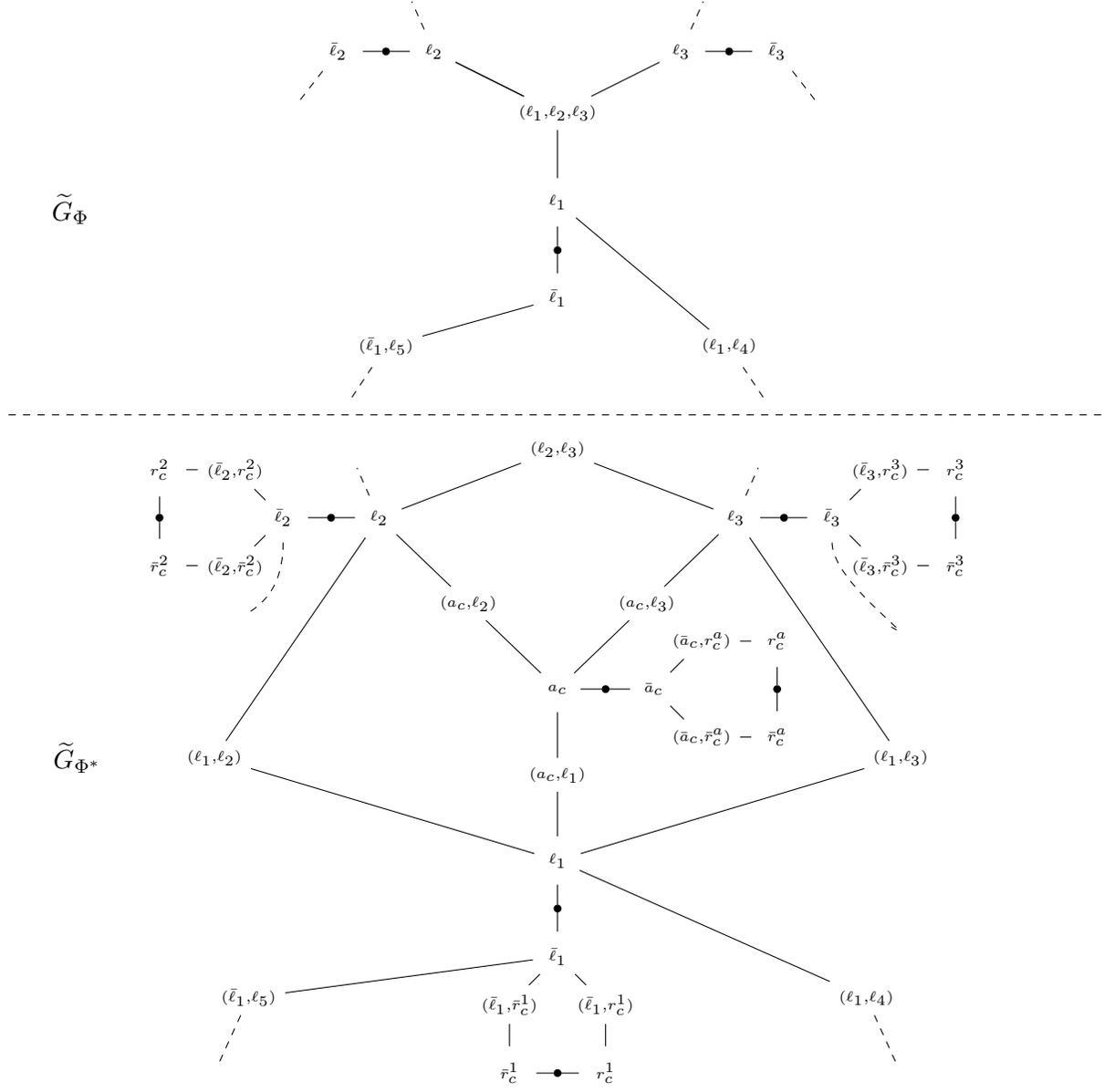


 \begin{claim}\label{claim:6or7}
  	For every $c\in C_3$, given any truth assignment of literals $\ell_1,\ell_2,\ell_3$, if $c$ is not satisfied then at most six clauses in   ${\cal C}_c$ can be simultaneously satisfied. If $c$ is satisfied, then there is a truth assignment of variable $a_c$ such that  seven clauses in  ${\cal C}_c$ are satisfied. 
 \end{claim}
\begin{proof}
Suppose $c$ is not satisfied, then either $a_c$ is True and the four clauses $(\ell_1,\ell_2)$, $(\ell_1,\ell_3)$, $(\ell_2,\ell_3)$, $(\bar{a_c})$ are unsatisfied or $a_c$ is False and the six clauses $(\ell_1,\ell_2)$, $(\ell_1,\ell_3)$, $(\ell_2,\ell_3)$, $(\ell_1,a_c)$, $(\ell_2,a_c)$, $(\ell_3,a_c)$ are unsatisfied.

Suppose now one literal \-- say $\ell_1$ \-- is True while $\ell_2,\ell_3$ are False, then choosing $a_c$ True  leaves only three unsatisfied clauses $(\bar {\ell_1}), (\ell_2,\ell_3)$ and $(\bar{a_c})$.

Suppose now that $\ell_1, \ell_2$ are True while $\ell_3$ is False, then any truth value for $a_c$ leaves only three unsatisfied clauses, $(\bar {\ell_1}), (\bar {\ell_2})$ and either $(\ell_3,a_c)$ or $(\bar{a_c})$.

Finally if the three literals $\ell_1, \ell_2, \ell_3$ are True, then choosing $a_c$ False leaves only three unsatisfied clauses, $(\bar {\ell_1}), (\bar {\ell_2})$ and $(\bar {\ell_3})$.

All other configurations are symmetrical.
\end{proof}
 
 As a consequence of Claims~\ref{claim:max7} and~\ref{claim:6or7}, any satisfied clause in $C_3$ will generate seven satisfied clauses in the new instance while unsatisfied clauses in $C$ will generate at most six satisfied clauses in the new instance. We need a similar property for 2-clauses. To this aim, we add for such a clause a system of 2-clauses on new variables  with always exactly six  \newpier{clauses} simultaneously satisfied.

~

 \underline{Case of 2-clauses}:
 
 
  For every 2-clause $c=(\ell_1, \ell_2)$, where $\ell_i$, $i=1,2$ are literals associated with variables in $X$, we introduce three  new variables $x_c,y_c,z_c$ and we add the set of eight 2-clauses ${\cal C}_c=\{(x_c,y_c), (x_c,\bar{y_c}), (x_c,z_c), (x_c,\bar{z_c}), (\bar{x_c},y_c), (\bar{x_c},\bar{y_c}), (\bar{x_c},z_c), (\bar{x_c},\bar{z_c})\}$ (see Figure~\ref{fig:H}).
 
\begin{claim}\label{claim:2clauses}
 For every $c\in C_2$, for all truth assignment\newpier{s} for variables $x_c,y_c,z_c$, exactly six clauses in ${\cal C}_c$ are satisfied.
\end{claim}
\begin{proof}
Suppose $x_c$ is True, then the four clauses including $x_c$  are True and exactly two clauses among $(\bar{x_c},y_c), (\bar{x_c},\bar{y_c}), (\bar{x_c},z_c), (\bar{x_c},\bar{z_c})$ are satisfied. Symmetrically, if $x_c$ is False we have exactly six satisfied clauses in ${\cal C}_c$.
 \end{proof}

 \begin{figure}[hbtp]
  \begin{center}
    \begin{tikzpicture}[node distance=1cm and 1cm]

\tikzstyle var2=[shape=circle]
\tikzstyle clause2=[shape=rectangle]

 \node[style = var2] (uc) at (0,0) {$\scriptstyle x_c$};
  \node[style = var2] (buc) at (3,0) {$\scriptstyle \bar x_c$};
\node[style = var2] (vc) at (0,2) {$\scriptstyle y_c$};
  \node[style = var2] (bvc) at (3,2) {$\scriptstyle \bar y_c$};
  
  \node[style = var2] (wc) at (0,-2) {$\scriptstyle z_c$};
  \node[style = var2] (bwc) at (3,-2) {$\scriptstyle \bar z_c$};

 \node[style = clause2] (ucvc) at (0,1) {$\scriptstyle (x_c,y_c)$};
  \node[style = clause2] (ucwc) at (0,-1) {$\scriptstyle ( x_c,z_c)$};
  \node[style = clause2] (bucbvc) at (3,1) {$\scriptstyle (\bar x_c,\bar y_c)$};
  \node[style = clause2] (bucbwc) at (3,-1) {$\scriptstyle ( \bar x_c,\bar z_c)$};
  \node[node2] (ucbuc) at (1.5,0) {};
  \node[node2] (vcbvc) at (1.5,2) {};
  \node[node2] (wcbwc) at (1.5,-2) {};
   \node[style = clause2] (ucbvc) at (1.5,1) {$\scriptstyle (x_c,\bar y_c)$};
   \node[style = clause2] (ucbwc) at (1.5,-1) {$\scriptstyle (x_c,\bar z_c)$};
   
\node[style = clause2] (bucvc) at (4,3) {$\scriptstyle (\bar x_c,y_c)$};
\node[style = clause2] (bucwc) at (4,-3) {$\scriptstyle (\bar x_c,z_c)$};


\foreach \from/\to in {uc/ucbuc,ucbuc/buc,vc/vcbvc,vcbvc/bvc,wc/wcbwc,wcbwc/bwc,uc/ucvc,ucvc/vc,uc/ucbvc,ucbvc/bvc,uc/ucwc,ucwc/wc,uc/ucbwc,ucbwc/bwc,buc/bucbvc,bucbvc/bvc,buc/bucbwc,bucbwc/bwc}
\draw[] (\from) -- (\to) ;

\draw (buc) to [out=45,in=-20] (bucvc);
\draw (bucvc) to [out=160,in=45] (vc);
\draw (buc) to [out=-45,in=20] (bucwc);
\draw (bucwc) to [out=-160,in=-45] (wc);
      \end{tikzpicture}
  \end{center}
  \caption{\mar{Gadget $H_c$ associated with a 2-clause $c\in C_2$.}}
  \label{fig:H}
\end{figure}
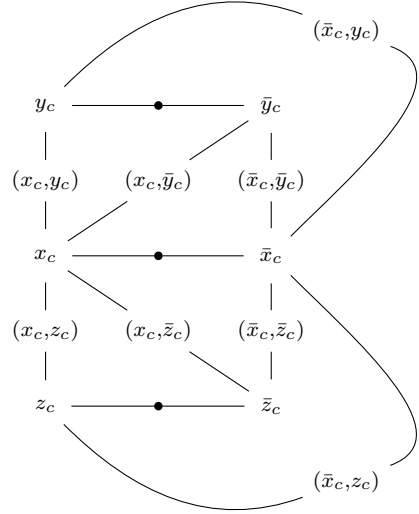


~

Denote by $X'$ the new set of variables and by $C'$ the new set of clauses, we then have:

\begin{claim}\label{claim:XX'}
	There is a truth assignment for variables in $X'$ satisfying at least $7|C|$ clauses in $C'$ if and only if there is a truth assignment for variables in $X$ satisfying all clauses in $C$. Moreover at most $7|C|$ clauses in $C'$ can be simultaneously satisfied.
\end{claim}

\begin{proof}
Suppose there is a truth assignment of variables in $X$ satisfying all clauses in $C$. Claims~\ref{claim:max7},  \ref{claim:6or7} and~\ref{claim:2clauses} ensure that there is a truth assignment of variables in $X'$ satisfying  exactly $7|C|$ clauses in $C'$. Conversely, consider a truth assignment of variables in $X'$; it defines a truth assignment of variable in $X$ since $X\subset X'$ (we keep the same values). Suppose this truth assignment does not satisfy a clause in $C$; then Claims~\ref{claim:max7},  \ref{claim:6or7} and~\ref{claim:2clauses} ensure that at most $7(|C|-1)+6=7|C|-1$ clauses of $C'$ are satisfied.
\end{proof}

To conclude the reduction we need to describe how to manage 1-clauses introduced when replacing 3-clauses in $C_3$:\\
 
\underline{Case of 1-clauses}:
 
For any 1-clause $c=(\ell)$, where $\ell$ is a literal corresponding to a variable in $X'$,  we add a new variable $r_c$ and replace the clause $(\ell)$ by the set of two 2-clauses $C_c^1=\{(\ell,r_c),(\ell,\bar{r_c})\}$. If $\ell$ is True then these two new clauses are satisfied while if   $\ell$ is False any truth alignment for $r_c$ satisfies only one of the two clauses in $C_c^1$. 
 
Let $C'_1$ be the set of 1-clauses in $C'$: 
$|C'_1|=4|C_3|$. We replace each clause in $C'_1$ by two 2-clauses as described above and denote by $X^\ast$ and $C^\ast$ the new set of variables and clauses.\\ 

\underline{The instance $(\Phi^\ast,K)$}:  
  
Putting all together, if $R$ denotes the set of variables $r_c$ added for each 1-clause $c\in C'_1$, we have:
 
\begin{equation}\label{eq:new var clauses}
 \begin{split}
 	X^\ast &= X\cup \{a_c ~|~ c\in C_3\}\cup\{u_c, v_c, w_c ~|~ c\in C_2\}\cup R\\
 C^\ast &= (\bigcup_{c\in C} C_c)\setminus C'_1 \cup \bigcup_{c\in C'_1}C_c^1
 \end{split}
\end{equation}

To finalize the reduction, we set $K=7|C|+4|C_3|$. $X^\ast, C^\ast, K$ define the instance $(\Phi^\ast,K)$. \gab{
We will show later that it is an instance of \RSP}.
 
The above discussion immediately shows:
 
\begin{claim}\label{claim:general}
 	There is a truth assignment of variables in $X$ satisfying all clauses in $C$ if and only if there is a truth assignment of variables in $X^\ast$ satisfying at least $K$ clauses in $C^\ast$.
\end{claim}
\begin{proof}

Assume first there is a truth assignment of variables in $X$ satisfying all clauses in $C$. We have seen that for each satisfied 3-clause $c\in C_3$ there is a truth assignment for the related variable $a_c$ inducing 7 satisfied clauses in $C'$. Moreover, each clause of the four related 1-clauses $c_1\in C'_1$ induces two clauses in $C^1_{c_1}$ both satisfied if $c_1$ is satisfied and only one satisfied otherwise, for any truth assignment of the related variables in $R$. Furthermore, for every satisfied  2-clause $c\in C_2$ and for all truth assignment\newpier{s} of variables $u_c, v_c, w_c$, $c$ as well as the six 2-clauses in $C_c$ are satisfied. In all it makes \ales{at least}
 $7|C|+4|C_3|=K$ satisfied clauses in $C^\ast$.
 
 Conversely, consider a truth assignment of variables in $X^\ast$ such that at least $K$ clauses are satisfied and consider the induced truth assignment for variables in $X\subset X^\ast$. If one clause $c_0\in C$ is not satisfied, then we have seen that at most $7|C|-1$ clauses in $C'$ are satisfied and even with the last $4|C_3|$ clauses in $X^\ast\setminus C'$ it would make at most $K-1$ satisfied clauses in $C^\ast$. So, all clauses in $C$ are satisfied by this truth assignment.
 \end{proof}\\
 
 Moreover Relation~\ref{eq:new var clauses} guarantees that the construction from $(X,C)$ to $(X^\ast,C^\ast,K)$ can be performed in polynomial time. So, to conclude the proof we need to show that $(X^\ast,C^\ast,K)=(\Phi^\ast,K)$ is an instance of \RSP.
\\

\underline{$(\Phi^\ast,K)$ is an instance of \RSP}:

Let us justify that every literal in $\Phi^\ast$ appears in at most four clauses. Consider first a variable $x\in X$. In $C$ each of the two relative literals $x$ and $\bar x$  appears in one 2-clause  in $C_2$: $c_{2,x}$ and  $c_{2,\bar x}$, respectively,  and either $x$ or $\bar x$ \-- say without loss of generality $x$ \-- appears in one 3-clause $c_{3,x} \in C_3$. Then, in $C^\ast$, $x$ will appear in the clause $c_{2,x}$ as well as in three 2-clauses in $C_{c_{3,x}}$ while $\bar x$ appears in three clauses of $C^\ast$:  the clause $c_{2,\bar x}$ as well as two clauses  $(\bar x, r_{(\bar x)}), (\bar x, \bar{r}_{(\bar x)})$ associated with the 1-clause $(\bar x) \in C_{c_{3,x}}$.
 
Consider now a variable $a_c, c\in C_3$: the literal $a_c$ appears in three 2-clauses of $C_c$ while $\bar {a_c}$ appears in the two clauses $(\bar{a_c}, r_{(\bar {a_c})}), (\bar{a_c}, \bar{r}_{(\bar{a_c})})$ associated with the 1-clause $(\bar{a_c})\in C_c$.
 
For every 2-clause $c\in C_2$, the literals $u_c$ and $\bar{u_c}$ both appear  in four clauses in $C_c$ while $v_c, \bar{v_c}, w_c, \bar{w_c}$ appear in only two clauses in $C_c$. Finally, each literal among $r_c$ and $\bar{r_c}$, $c\in C'_1$, corresponding to variables in $R$ appears in a single 2-clause in $C^\ast$.\\
 
To conclude the proof, \mar{note that $\widetilde {G}_{\Phi^\ast}$ is of maximum degree~5 and  we need to show that 
it} is planar. As mentioned previously since every variable appears in three clauses in $\Phi$, $\widetilde G_\Phi$ is planar. $\widetilde {G}_{\Phi^\ast}$ is obtained from $\widetilde G_\Phi$ in three steps, each preserving planarity.


\mar{First consider a  3-clause $c=(\ell_1, \ell_2, \ell_3)$ in $C$ and assume for instance that the literal $\ell_1$ appears in the 2-clause $(\ell_1,\ell_4)$ while $\bar\ell_1$ appears in the 2-clause $(\bar\ell_1,\ell_5)$. We modify $\widetilde G_\Phi$ as follows. Remove from $\widetilde G_\Phi$ the clause-vertex $(\ell_1, \ell_2, \ell_3)$ and its three incident edges towards the literal-vertices  $\ell_1$, $\ell_2$ and $\ell_3$;  add a path  $a_{c}a'_{c}\bar{a_{c}}$, new clause vertices  $(a_{c},\ell_1)$, $(a_{c},\ell_2)$, $(a_{c},\ell_3)$ and the six edges   $a_{c}(a_{c},\ell_i), (a_{c},\ell_i)\ell_i$, $i=1,2,3$.  Add the three clause-vertices  $(\ell_1,\ell_2), (\ell_1,\ell_3)$ and $(\ell_2,\ell_3)$, respectively and  add the six edges $(\ell_i,\ell_j)\ell_i,(\ell_i,\ell_j)\ell_j, 1\leq i<j\leq 3$. 
Then,  consider the  1-clauses $(\ell)\in C'_1$ added for the clause $c=(\ell_1, \ell_2, \ell_3)$: $\ell\in \{\bar\ell_1, \bar\ell_2, \bar\ell_3, \bar a_c\}$. Denote $r^i_c, i=1,2,3$ the $r$-variable associated with the 1-clause $(\bar\ell_i), i=1,2,3$ and $r^a_c$ the $r$-variable associated with the 1-clause $(\bar a_c)$. Add in $\widetilde G_\Phi$ a cycle on 6 vertices for each of these 1-clauses:, $\bar\ell_i \-- (\bar\ell_i,r^i_c) \-- r^i_c \-- r'^i_c\-- \bar r^i_c\-- (\bar\ell_i,\bar r^i_c)\--\bar\ell_i$ for $i=1,2,3$ and $\bar a_c \-- (\bar a_c,r^a_c) \-- r^a_c \-- r'^a_c\-- \bar r^a_c\-- (\bar a_c,\bar r^a_c)\--\bar a_c$.   
(see Figure~\ref{fig:3-clause}). As illustrated in the figure, this transformation preserves planarity since the added gadget is planar and all dashed edges linking it to the rest of the graph can be preserved without crossing other edges. }

 
 
 Second, for every 2-clause $c\in C_2$ we add in $\widetilde {G}_{\Phi^\ast}$  an independent copy of the planar graph $H_c$ associated with variables $x_c,y_c,z_c$  and depicted in Figure~\ref{fig:H}. All graphs $H_c, c\in C_2$ are isomorphic. They are planar as a subdivision of the graph obtained from two copies of a clique $K_4$ glued by identifying one edge.
 
 Since all these two transformations preserve planarity and since $\widetilde G_\Phi$  is planar,  $\widetilde G_{\Phi^\ast}$ is planar, which completes the proof of Proposition~\ref{prop:complex-Max 2SAT}.  
 \end{proof}

\subsection{Hardness \mar{of \MSFS} in planar graphs}
In what follows, we \mar{consider \MSFS. We recall it is defined on a mixed graph and corresponds to the case where all spread probabilities are~1.}  
Since \MSFS is a particular case of \MFS, all hardness results for the former apply to the latter.

\mar{Note first  that the computation of $\rho$ is in general $\sharp P$-hard~\cite{wang2012scalable}, but when the probability of spread is set to 1 for each edge,  
\mar{$\rho(G_H)$ is given in Lemma~\ref{lem:px}.} 
Clearly, this calculation can be performed in polynomial time. This means that, given a cut system $H$ we can check in polynomial time if it is a solution. As a consequence, contrary to \MFS, \MSFS is in NP.}

\begin{lem}\label{lem:NP}
\MSFS is in NP.
\end{lem}

\mar{Let us also note that, if we allow any value for edge costs and vertex values, then a simple reduction from {\sc Partition} shows that \MSFS is NP-complete on stars:}

\begin{prop}\label{pro:partition}
\mar{\MSFS is NP-complete on stars.}
\end{prop}

\begin{proof}
\mar{Lemma~\ref{lem:NP} shows it is in NP.
Consider an instance of {\sc Partition} consisting of $n$ integers $s_1, \ldots, s_n$ with $\sum\limits_{i=1}^n s_i=2S$ for an integer $S$. The question is whether there is a subset $A\subset \{1, \ldots,n\}$ such that 
$\sum\limits_{i\in A} s_i=S$. {\sc Partition} is known to be NP-complete~\cite{Garey1979ComputersNP-completeness}. }

\mar{We build an instance of \MSFS as follows: consider, as $G$, a star with a center $o$ and $n$ leaves $\ell_1, \ldots,\ell_n$. The values are defined as follows: $\varphi(o)=0$ and $\varphi(\ell_i)=s_i, i=1, \ldots,n$ and the costs of edges are given by $\kappa(o\ell_i)=s_i, i=1, \ldots,n$. Finally $o$ has ignition probability~1 and other vertices have a probability of ignition~0. We choose $B=R=S$.}

\mar{With this set-up, a cut system $H$ is defined as $H=\{o\ell_i, i\in A\subset\{1, \ldots,n\}\}$, $\kappa(H)=\sum\limits_{i\in A}s_i$ and $\rho(G_H)= \sum\limits_{i\notin A}s_i$. This instance of \MSFS is a yes-instance if and only if there is a set $A\subset\{1, \ldots,n\}$ such that $\sum\limits_{i\in A}s_i\leq S$ and $\sum\limits_{i\notin A}s_i\leq S$. Since $\sum\limits_{i\notin A}s_i+\sum\limits_{i\in A}s_i=2S$, the instance of \MSFS is a yes-instance if and only if the instance of {\sc Partition} in a yes-instance. This completes the proof.}
\end{proof}


\gab{Since {\sc Partition} is weakly NP-complete, the previous reduction does not give any information about complexity when edge costs and vertex values are polynomially bounded. 
In what follows, we focus on this 
 case for the \MSFS problem. 
In Proposition~\ref{th:complete-planar}, \mar{we establish its NP-completeness in bipartite planar graphs with polynomially bounded edge costs and vertex values.} The proof is based on a polynomial time reduction from \RSP. } 

An immediate consequence of Proposition~\ref{pro:partition}  is that \MFS is NP-hard.

\begin{lem}\label{lem:basic}
Let $G=(V,E)$ be a graph consisting in two components isomorphic to $P_3$. For each $u\in V$, let $\varphi(u)=\nu$ and $\pi_i(u)=p$ if $deg(u)=1$, otherwise $\pi_i(u)=q$ if $deg(u)=2$. Let $\pi_s(e)=1$ for each  $e\in E$. Let $H_1$ be a cut system consisting of two edges from the same component and let $H_2$ be a cut system of two edges one from each component. Then $\rho(G_{H_2})< \rho(G_{H_1})$ if $0<p<\frac 2 3$, $0\leq q < 1$, and \mar{$\rho(G_{H_1}) - \rho(G_{H_2})=p(2-3p)(1-q)\nu$}.
\end{lem}

\begin{proof}
The initial graph $G$ and the two graphs resulting from the application of the cut systems $H_1$ and $H_2$ are depicted in Figure \ref{fig:twocutsystems}. Let us calculate the risks $\rho(G_{H_1})$ and $\rho(G_{H_2})$, \mar{as a direct application of Lemma~\ref{lem:px}}:

\begin{align*}
\rho(G_{H_1}) & = \nu(2p+q) +3\nu(1-(1-p)^2(1-q)) \\
\rho(G_{H_2}) & = 2\nu(p+2(1-(1-p)(1-q)))
\end{align*}

By dividing the risks by the node value $\nu$ and then expanding the equations, we obtain:

\begin{align*}
\frac{\rho(G_{H_1})}{\nu} & = 2p+q+3(2p+q-p^2-2pq+p^2q) = 8p+4q-3p^2-6pq+3p^2q \\
\frac{\rho(G_{H_2})}{\nu} & = 2(p+2(p+q-pq))=6p+4q-4pq
\end{align*}

The difference between the two risks, normalized by $\nu$, is then:

\begin{align*}
\frac{\rho(G_{H_1})}{\nu} - \frac{\rho(G_{H_2})}{\nu} = 2p-3p^2-2pq+3p^2q = p(2-3p)(1-q)
\end{align*}

Such a difference is strictly positive and therefore $\rho(G_{H_2}) < \rho(G_{H_1})$ when:

\begin{align*}
0 < p < \frac{2}{3}\\
0 \leq q < 1
\end{align*}

\end{proof}

\begin{figure}[tb]
  \begin{center}
    \begin{tikzpicture}[node distance=1cm and 1cm]
    \node[rectangle, minimum width=2cm] (g) {$G$};

    \node[node, label=below:$p$, label=above:$\nu$] (xl11) [right=of g]    {};
    \node[node, label=below:$q$, label=above:$\nu$] (v11)  [right=of xl11] {};
    \node[node, label=below:$p$, label=above:$\nu$] (xr11) [right=of v11]  {};

    \node[] (plus1) [right=of xr11] {};

    \node[node, label=below:$p$, label=above:$\nu$] (xl21) [right=of plus1] {};
    \node[node, label=below:$q$, label=above:$\nu$] (v21)  [right=of xl21]  {};
    \node[node, label=below:$p$, label=above:$\nu$] (xr21) [right=of v21]   {};

    \path [line] (xl11) -- node [ midway,above ] {$1$} (v11);
    \path [line] (v11) -- node [ midway,above ] {$1$} (xr11);
    \path [line] (xl21) -- node [ midway,above ] {$1$} (v21);
    \path [line] (v21) -- node [ midway,above ] {$1$} (xr21);

    \draw [dashed] (-1,-0.9) -- (10.5,-0.9);

    \node[rectangle, minimum width=2cm] (g1) [below=of g] {$G_{H_1}$};

    \node[node] (xl12) [right=of g1]   {};
    \node[node] (v12)  [right=of xl12] {};
    \node[node] (xr12) [right=of v12]  {};

    \node[] (plus2) [right=of xr12] {};

    \node[node] (xl22) [right=of plus2] {};
    \node[node] (v22)  [right=of xl22]  {};
    \node[node] (xr22) [right=of v22]   {};

    \draw[-] (xl22.east) -- (v22.west);
    \draw[-] (v22.east) -- (xr22.west);

    \draw [dashed] (-1,-2.45) -- (10.5,-2.45);

    \node[rectangle, minimum width=2cm] (g2) [below=of g1] {$G_{H_2}$};

    \node[node] (xl13) [right=of g2]   {};
    \node[node] (v13)  [right=of xl13] {};
    \node[node] (xr13) [right=of v13]  {};

    \node[] (plus3) [right=of xr13] {};

    \node[node] (xl23) [right=of plus3] {};
    \node[node] (v23)  [right=of xl23]  {};
    \node[node] (xr23) [right=of v23]   {};

    \draw[-] (v13.east) -- (xr13.west);
    \draw[-] (v23.east) -- (xr23.west);
    \end{tikzpicture}
  \end{center}
  \caption{Graphs $G$, $G_{H_1}$ and $G_{H_2}$.\mar{Vertices in $G$ (up) are indicated with their probability of ignition (below) and their value (above).}}
  \label{fig:twocutsystems}
\end{figure}
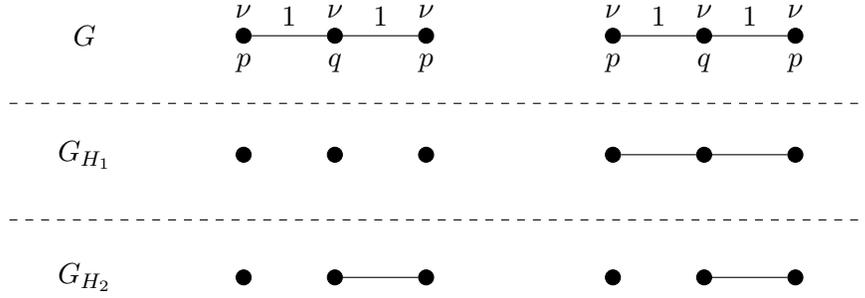

\begin{prop}\label{th:complete-planar}
\MSFS is NP-complete on bipartite planar graphs of maximum degree~5 \mar{and with polynomially bounded vertex values and edge costs}.	
\end{prop}
\begin{proof}
\MSFS is in NP according to Lemma~\ref{lem:NP}.


\mar{In the figures given in this proof, \ales{vertices} are labeled  with their name (unless unnecessary for a good understanding) and the related ignition probability, when this information is useful. Edges will be labeled with the related probability of spread, when useful. To simplify the figures, the values of vertices and the costs of edges are not reported but directly indicated in the text, when relevant. Finally, for a variable $z$, $l_z\in\{z, \bar z\}$ will denote a literal on $z$.}

\bigskip

Let $I=(\Phi, K)$, with $\Phi=(X,C)$, be an instance of \RSP with $n$ variables and $m$ 2-clauses. 

\bigskip

We build an instance $(\Gamma^I,\mar{\mathds{1}},\pi_i,\kappa,\varphi,B,R)$  of \MSFS, starting from $(\Phi, K)$ and the related graph $\widetilde{G}_\Phi$, as follows.




\begin{figure}[tb]
\centering
\begin{tikzpicture}[node distance=1cm and 1.8cm]
\node[node,label=$\bar{x}$, label=below:$\frac{1}{2}$]  (xl) {};
\node[node,label=$x'$, label=below:$\frac{1}{2}$]       (xc) [right=of xl] {};
\node[node,label=$x$, label=below:$\frac{1}{2}$]        (xr) [right=of xc] {};

\node[rectangle, minimum width=1cm][below=of xc](a) {(a) Variable $P_3$};
\node[rectangle, minimum width=1cm][below= 3 pt of a] {Vertices have value $\nu$ and edges cost $s$};
\path [line] (xl.east) -- node [above,midway] {$1$} (xc.west);
\path [line] (xc.east) -- node [above,midway] {$1$} (xr.west);
\end{tikzpicture}
\hspace{1cm}
\begin{tikzpicture}[node distance=1cm and 1.8cm]
\node[node, label=$l_x^c$, label=below:$\frac{1}{2}$]  (xl)               {};
\node[node,label = $c'$, label=below:$q$]               (v)  [right=of xl] {};
\node[node, label=$l_y^c$, label=below:$\frac{1}{2}$]  (xr) [right=of v]  {};
 
 \node[rectangle, minimum width=1cm][below=of v](b){(b) Clause $P_3$};
 \node[rectangle, minimum width=1cm][below= 3pt of b] {Vertices have value $\omega$ and edges cost $1$};
\path [line] (xl.east) -- node [above,midway] {$1$} (v.west);
\path [line] (v.east) -- node [above,midway] {$1$} (xr.west);
\end{tikzpicture}

\caption{\mar{Representation of (a) a variable $x$  and (b) a clause $c=(l_x^c,l_y^c)$} in an instance of \MSFS.} 
\label{fig:varclauses}
\end{figure}

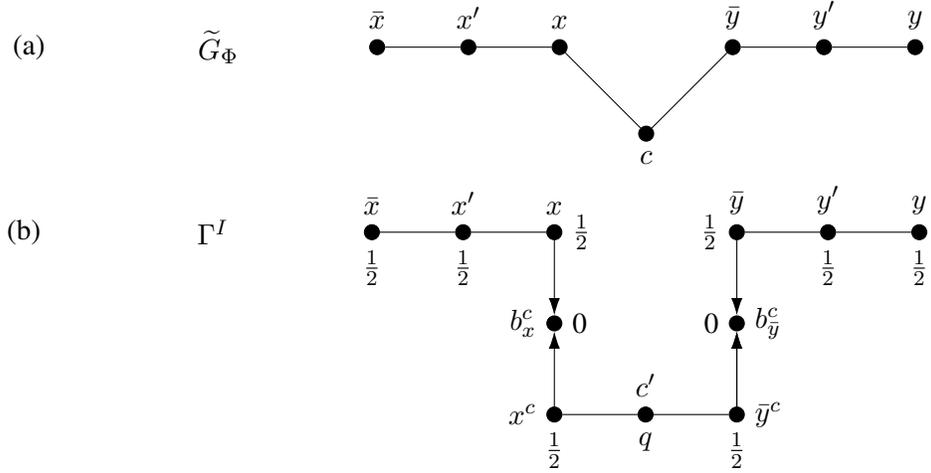
\begin{figure}[tb]
\centering
\begin{tikzpicture}

\node[rectangle, minimum width=2cm] (gtilde) [left=of xl] {$\widetilde{G}_\Phi$};
\node[rectangle, minimum width=1cm] (a) [left=of gtilde] {(a)};
\node[node, label=$\bar{x}$] (xl)               {};
\node[node, label=$x'$]      (xc) [right=of xl] {};
\node[node, label=$x$]       (xr) [right=of xc] {};
\draw[-] (xl.east) -- (xc.west);
\draw[-] (xc.east) -- (xr.west);

\node[node, label=below:$c$] (c) [below right=of xr] {};
\draw[-] (xr) -- (c);

\node[node, label=$\bar{y}$] (yl) [above right=of c] {};
\node[node, label=$y'$]      (yc) [right=of yl] {};
\node[node, label=$y$]       (yr) [right=of yc] {};
\draw[-] (yl.east) -- (yc.west);
\draw[-] (yc.east) -- (yr.west);

\draw[-] (yl) -- (c);
\end{tikzpicture}
\\
\begin{tikzpicture}
\node[rectangle, minimum width=2cm] (gammai) [left=of xl] {$\Gamma^I$};
\node[rectangle, minimum width=1cm] (b) [left=of gammai]{(b)};
\node[node, label=$\bar{x}$,label=below:$\frac{1}{2}$ ] (xl)               {};
\node[node, label=$x'$,label=below:$\frac{1}{2}$]      (xc) [right=of xl] {};
\node[node, label=$x$,label=right:$\frac{1}{2}$]       (xr) [right=of xc] {};
\draw[-] (xl.east) -- (xc.west);
\draw[-] (xc.east) -- (xr.west);

\node[node,label=left:$b_x^c$,label=right:0] (bx) [below=of xr] {};

\node[node, label=left:$x^c$,label=below:$\frac{1}{2}$] (lxc) [below=of bx]  {};
\node[node, label=above:$c'$,label=below:$q$]    (cp)  [right=of lxc] {};
\node[node, label=right:$\bar y^c$,label=below:$\frac{1}{2}$] (lyc) [right=of cp]  {};
\draw[-] (lxc.east) -- (cp.west);
\draw[-] (cp.east) -- (lyc.west);

\draw[->] (xr.south) -- (bx.north);
\draw[->] (lxc.north) -- (bx.south);

\node[node,label=right:$b_{\bar y}^c$,label=left:0] (by) [above=of lyc] {};
\draw[->] (lyc.north) -- (by.south);

\node[node, label=$\bar{y}$, label=left:$\frac{1}{2}$] (yl) [above=of by] {};
\node[node, label=$y'$,label=below:$\frac{1}{2}$]      (yc) [right=of yl] {};
\node[node, label=$y$,label=below:$\frac{1}{2}$]       (yr) [right=of yc] {};
\draw[-] (yl.east) -- (yc.west);
\draw[-] (yc.east) -- (yr.west);

\draw[->] (yl.south) -- (by.north);
\draw[->] (lyc.north) -- (by.south);

\end{tikzpicture}

\caption{Representation of clause $(x,\bar{y})$ in $\widetilde{G}_\Phi$ (a) and the corresponding $\Gamma^I$ (b).  \mar{Ignition probabilities of vertices are indicated on $\Gamma^I$. Probabilities of spread are all~1, $\varphi(b_{x}^c) =\varphi(b_{\bar y}^c)=1$ and $\kappa(xb_x^c)=\kappa(x^cb_x^c)=\kappa(\bar yb_{\bar y}^c)=\kappa(\bar y^cb_{\bar y}^c)=B+1$. All other vertex values and edge costs in $\Gamma^I$ are as in Figure~\ref{fig:varclauses}. } }
\label{fig:rspmfs}
\end{figure}

\paragraph*{Variables} \mar{In $\Gamma^I$, each variable $x$ is represented in $\widetilde{G}_\Phi$ by a path $P_3$ $xx'\bar x$} (see Figure~\ref{fig:varclauses}-a). For each vertex $u$ of \mar{these $P_3$s}, we set \mar{$\pi_i(u)=\frac 1 2$ and $\varphi(u)=\nu$}. For each edge $e$ of \mar{these $P_3$s}, we set $\pi_s(e)=1$ and $\kappa(e)=s$. As in $\widetilde{G}_\Phi$, for each variable $x\in X$, the external vertices of the corresponding 
$P_3$ represent the literals $x$ and $\bar x$.
\mar{We will refer to these paths as {\em variable $P_3$s}}.

\paragraph*{Clauses} As for the clauses, for each vertex $c$ of $\widetilde{G}_\Phi$ representing a clause $c=(l_x^c,l_y^c)\in C$, there is a \mar{path $P_3$ $l_x^cc'l_y^c$ in $\Gamma^I$} whose external vertices represent the literals $l_x^c$ and $l_y^c$ (see Figure~\ref{fig:varclauses}-b).  For each vertex $u$ of  \mar{these $P_3$s} 
we define $\varphi(u)=\omega$, whereas $\pi_i(u)=\frac 1 2$ if $u$ is an external vertex of \mar{the corresponding} $P_3$, $\pi_i(u)=q$ otherwise. For each edge $e$ of these $P_3$s, we set $\pi_s(e)=1$ and $\kappa(e)=1$.
\mar{We will denote these paths {\em clause $P_3$s}}.

\paragraph*{Connection between \mar{literals} and clauses.} Variables and clauses are then connected as follows. For each edge $cl_x$  
in $\widetilde{G}_\Phi$, \mar{where   $l_x\in \{x, \bar x\}$  
is a literal appearing} in $c$, \mar{we introduce} a \emph{binding} vertex 
\mar{$b_{l_x}^c$} in $\Gamma^I$ such that 
\mar{$\pi_i(b_{l_x}^c)=0$ and $\varphi(b_{l_x}^c) =1$}
Then, for each binding vertex $b_{l_x}^c$, \mar{with $l_x\in \{x, \bar x\}$}, there are two directed 
edges \mar{$e_1=l_x b_{l_x}^c$ and $e_2=l_x^cb_{l_x}^c$} 
in $\Gamma^I$ such that $\pi_s(e_1)=\pi_s(e_2)=1$ and $\kappa(e_1)=\kappa(e_2)=B+1$. See Figure~\ref{fig:rspmfs} for an example of 
graphs 
\mar{$\widetilde{G}_\Phi$ and $\Gamma^I$} when $C$ 
\mar{includes} only the clause $c=(x,\bar y)$. Note that all the probabilities of spread are set to 1 and then the built instance \ales{is} 
\mar{a \MSFS instance}. Note that the related graph is planar and bipartite of maximum degree~5.

\medskip

To conclude the transformation, we set $B=ns+m$ and $R=2n\nu+m\omega\left(\frac 3 2 + q\right) + m\left(\frac 7 4 + \frac q 8\right)- \frac K 8$. \mar{Note that, since $K\leq m$, we have $R>0$.}\\

\bigskip

We now show that, for certain values of the 
\mar{parameters $s,\nu,\omega$} and $q$, $I$ is satisfiable if and only if there is a cut system $H$ for $\Gamma^I$ such that $\kappa(H)\leq B$ and $\rho(\Gamma^I_H)\leq R$.

First, we choose $s=m+1$; this way, any cut system for $\Gamma^I$ can have at most $n$ edges from $P_3$s representing the variables. This is due to the choice of $B$: 
a cut system with $n+1$ cuts on the $P_3$s representing the variables has a total cost of $(n+1)s$, but this is impossible since $(n+1)s=ns+s > ns+m=B$.

Second, we choose the value of $\nu$ in such a way that $n$ cuts on the \mar{$P_3$s} representing the variables are necessary to keep the total risk below $R$.

To this end, we assume that $n_i$ variable $P_3$s have $i$ cuts for $i=0,1,2$: $n_0+n_1+n_2=n$ and the related number of cuts is $n_1+2n_2$. If at most $n$ cuts are performed on variable $P_3$s, then $n_2\leq n_0$ and then, Lemma~\ref{lem:basic}  ensures that, while $n_2> 0$, pairing one $P_3$ with two cuts with one with no cut \mar{and transferring one cut from the former to the latter}  allows to reduce the risk without changing the number of cuts used on the $n$ variable $P_3$s. 

\mar{Assume first that no more than $n-1$ cuts are performed on those $P_3$s. Then, $n_0\geq n_2+1$ and the smallest contribution to the total risk induced by vertices from those $P_3$s is obtained for $n_2=0$, with still $n_0\geq 1$.} 
Then, Lemma~\ref{lem:px} \mar{ensures that the related} contribution to the risk is $\mar{(n-n_0)\nu\left(\frac 1 2 + 2\left(1- (\frac 1 2)^2 \right)\right) = 2\nu(n-n_0)}$ for 
\mar{the $(n-n_0)$ variable $P_3$s with one cut} and a contribution of 
$\mar{3n_0\nu(1-(\frac 1 2)^3)}= \frac {21} 8 n_0\nu$ for the \mar{$n_0$} remaining uncut paths. So, in summary, \mar{since $n_0\geq 1$}, the contribution to the total risk 
\mar{induced by the variable} $P_3$s is at least  $(2(n-1)+\frac {21} 8)\cdot \nu$.

Let us choose the value of $\nu$ such that 
\begin{equation}\label{eq:nu}
\left(2(n-1)+\frac {21} 8\right)\cdot \nu \geq 2n\nu+ m\omega\left(\frac 3 2 + 1\right) + 2m.
\end{equation}

\mar{Since} $2n\nu+ m\omega(\frac 3 2 + 1) + 2m > 2n\nu+m\omega\left(\frac 3 2 + q\right) + m\left(\frac 7 4 + \frac q 8\right)- \frac K 8 = R$, Equation~\ref{eq:nu} will \mar{then} imply $(2(n-1)+\frac {21} 8)\cdot \nu>R$. 

\mar{To satisfy Equation~\ref{eq:nu} we have}:

\begin{align*}
    &&\left(2(n-1)+\frac {21} 8\right)\cdot \nu &\geq 2n\nu+ m\omega\left(\frac 3 2 + 1\right) + 2m\\
   &\Leftrightarrow & 2n\nu + \frac 5 8 \nu &\geq 2n\nu+ \frac 5 2 m\omega + 2m \\
  &\Leftrightarrow  &\frac 5 8 \nu &\geq m\left(\frac 5 2 \omega + 2\right)\\
  & \Leftrightarrow& \nu &\geq \frac 8 5 m\left(\frac 5 2 \omega + 2\right) = 4 m\left(\omega+\frac 4 5\right)
\end{align*}

By setting $\nu=4 m\left(\omega+\frac 4 5\right)$, we are 
\mar{guaranteed} that any positive solution of \MFS uses \mar{exactly} $n$ cuts on the $P_3$s representing the variables in $\Gamma^I$. 

\mar{We now need to} guarantee that the $n$ cuts are distributed exactly as one for each $P_3$. \mar{We now have $n_0=n_2$ and} 
Lemma~\ref{lem:basic} \mar{shows that the risk induced by the variable $P_3$s decreases with $n_0$}. 
\mar{If $n_0=1$, then  Lemma~\ref{lem:px} shows that this risk is} $2\nu n+\frac 1 2 \left(2 - 3 \frac 1 2\right)\left(1-\frac 1 2\right)\nu= 2n\nu +\frac \nu 8$. Similarly as we 
\mar{did} before, 
\mar{we choose $\nu$ such that this expression is greater than $R$:}
\begin{align*}
   && 2n\nu +\frac \nu 8&\geq 2n\nu+ m\omega\left(\frac 3 2 + 1\right) + 2m\\
 &\Leftrightarrow &   \frac \nu 8 &\geq m\left(\frac 5 2 \omega + 2\right)\\
 &\Leftrightarrow &   \nu &\geq 8m\left(\frac 5 2 \omega + 2\right)
\end{align*}

\mar{Since} $8m\left(\frac 5 2 \omega + 2\right)\geq 4 m\left(\omega+\frac 4 5\right)$ for each positive $\omega$, we can set:

$$\nu=8m\left(\frac 5 2 \omega + 2\right)$$

In summary, with the aforementioned choices of $s$ and $\nu$, \mar{any cut system $H$  for $\Gamma^I$ such that $\kappa(H)\leq B$ and $\rho(\Gamma^I_H)\leq R$}  uses $n$ cuts on the edges of the variables, one for each variable. \mar{The related induced cost is} 
$ns$ and 
\mar{the induced} risk is $2n\nu$.

Then, the remaining budget 
$m$ can be used for $m$ cuts on the edges of the clauses, 
\mar{to maintain} the total risk under $R-2n\nu$. 
We want to force one cut for each  clause $P_3$. 

\mar{As before, we denote $m_i$ the number of clause $P_3$s with $i$  cuts, $i=0,1,2$. We have $m_2=m_0$ and Lemma~\ref{lem:basic} guarantees that, if $m_2>0$, then we can reduce the risk by replacing one clause $P_3$ with two cuts and one with no cut by two clause $P_3$s with one cut each. Thus, the minimum risk induced by clause $P_3$s with $m_2>0$ is obtained for $m_2=m_0=1$. Using Lemma~\ref{lem:px}, the risk induced when all $m$ clause $P_3$s have a single cut is 
$m\omega(\frac 1 2 + 2(\frac 1 2+q-\frac 1 2 q))=m\omega(\frac 3 2 +q)$. Using Lemma~\ref{lem:basic}, if two clause $P_3$s with one cut are replaced by one with two cuts and one without any cut, then the risk increase is
$\frac 1 2 \left(2 - 3 \frac 1 2\right)(1-q)\omega=\frac \omega 4(1-q)$.
}

We choose the value of $\omega$ in such a way that 
\mar{$ \frac \omega 4(1-q)\geq 2m \Leftrightarrow \omega \geq 2m\frac 4 {1-q}$ since in this case we will have:}
\begin{equation}\label{eq:clause}
    m\omega\left(\frac 3 2 +q\right) + \frac \omega 4(1-q) \geq m\omega\left(\frac 3 2 +q\right)+2m> R-2n\nu
\end{equation}



Then, 
\mar{choosing} $\omega= \frac {8m} {1-q}$ (the value of $q$ 
will be defined at the end of the proof) in addition to the previous choices of $s$ and $\nu$,  \mar{we are ensured that any cut system $H$  for $\Gamma^I$ such that $\kappa(H)\leq B$ and $\rho(\Gamma^I_H)\leq R$  uses one cut per variable $P_3$ and one per clause $P_3$.}\\

Let us \mar{now assume that the instance $I$ of} \RSP is satisfiable, i.e., the instance $I$ given by $(X,C,K)$, has an assignment for the variables in $X$ such that at least $K$ clauses are satisfied. Then, we prove that there exists a cut system $H$ for $\Gamma^I$ such that $\kappa(H)\leq B$ and $\rho(\Gamma^I_H)\leq R$.

With the aforementioned choices of $s,\nu$ and $\omega$, any cut system for $\Gamma^I$ has $n$ cuts on the variable edges (one for each path representing a variable) and $m$ cuts on the clause edges (one for each path representing a clause). Note that no cut is possible for the incoming edges of any binding vertex since 
\mar{their} cost is larger than the budget $B$.

The sought cut system $H$ cuts each 
$P_3$ \mar{associated with a variable} on the side of the true literal. With this assumption, we evaluate now the induced risk on the binding vertices.

We denote with $K_i$ the number of clauses with $i$ satisfied literals, $i=0,1,2$. Then $K_1+K_2\geq K$ and $K_0+K_1+K_2=m$.

Let $\rho_i$ be the induced risk on the binding vertices, for 
\mar{all clauses} with $i$ true literals. Then the total induced risk on the binding vertices is $\rho=\rho_0+\rho_1+\rho_2$. The subgraphs of $\Gamma^I_H$ representing a clause $c=(l_x,l_y)$ for all the possible assignments to $l_x$ and $l_y$ are shown in Figures~\ref{fig:clauseff}, \ref{fig:clausevf}, and \ref{fig:clausevv}. Note that a cut on a clause 
\mar{$P_3$} can be 
\mar{on any of the two edges}. When $l_x$ and $l_y$ have the same assignment, the subgraphs resulting by cutting \mar{the clause $P_3$} \gab{either} 
\mar{on the left-hand side or on the right-side} 
are isomorphic (see Figures~\ref{fig:clauseff} and~\ref{fig:clausevv}), whereas, if $l_x$ is False and $l_y$ is True, i.e., the clause has only one satisfied literal, \mar{then} the two \mar{possible} subgraphs, \mar{depending on which edge of the clause $P_3$ is cut}, are not isomorphic (see Figure~\ref{fig:clausevf}).


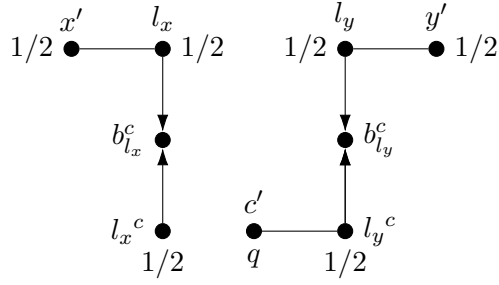
\begin{figure}[tb]
\centering
\begin{tikzpicture}
\node[node, label=$x'$, label=left:$1/2$] (xc)               {};
\node[node, label=$l_x$, label=right:$1/2$] (xr) [right=of xc] {};
\draw[-] (xc.east) -- (xr.west);

\node[node,label=left:$b_{l_x}^c$] (bx) [below=of xr] {};

\node[node, label=left:${l_x}^c$, label=below:$1/2$] (lxc) [below=of bx] {};
\node[node,label=above:$c'$, label=below:$q$]   (cp)  [right=of lxc]      {};
\node[node, label=right:${l_y}^c$, label=below:$1/2$] (lyc) [right=of cp]       {};
\draw[-] (cp.east) -- (lyc.west);

\draw[->] (xr.south) -- (bx.north);
\draw[->] (lxc.north) -- (bx.south);

\node[node,label=right:$b_{l_y}^c$] (by) [above=of lyc] {};
\draw[->] (lyc.north) -- (by.south);

\node[node, label=$l_y$, label=left:$1/2$] (yl) [above=of by] {};
\node[node, label=$y'$, label=right:$1/2$] (yc) [right=of yl] {};
\draw[-] (yl.east) -- (yc.west);

\draw[->] (yl.south) -- (by.north);
\draw[->] (lyc.north) -- (by.south);

\end{tikzpicture}
\caption{Clause $c=(l_x,l_y)$, $l_x=$ False, $l_y=$ False. \mar{Example where the clause $P_3$ is cut on the left-hand side.}}
\label{fig:clauseff}
\end{figure}

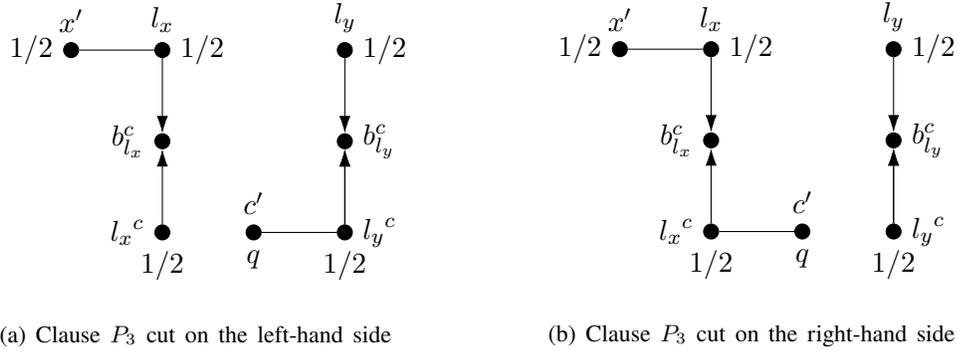
\begin{figure}[tb]
\centering
\begin{tikzpicture}
\node[node, label=$x'$, label=left:$1/2$]             (xc)               {};
\node[node, label=$l_x$, label=right:$1/2$] (xr) [right=of xc] {};
\draw[-] (xc.east) -- (xr.west);

\node[node,label=left:$b_{l_x}^c$] (bx) [below=of xr] {};

\node[node, label=left:${l_x}^c$, label=below:$1/2$] (lxc) [below=of bx]  {};
\node[node, label=above:$c'$,label=below:$q$]   (cp)  [right=of lxc] {};
\node[node, label=right:${l_y}^c$, label=below:$1/2$] (lyc) [right=of cp]  {};
\draw[-] (cp.east) -- (lyc.west);
\node[rectangle, minimum width=1cm, text width=.45\linewidth][below=of cp] {\footnotesize\mar{(a) Clause $P_3$  cut on the left-hand side}};
\draw[->] (xr.south) -- (bx.north);
\draw[->] (lxc.north) -- (bx.south);

\node[node,label=right:$b_{l_y}^c$] (by) [above=of lyc] {};
\draw[->] (lyc.north) -- (by.south);

\node[node, label=$l_y$, label=right:$1/2$] (yl)  [above=of by] {};

\draw[->] (yl.south) -- (by.north);
\draw[->] (lyc.north) -- (by.south);

\end{tikzpicture}
~
\begin{tikzpicture}
\node[node, label=$x'$, label=left:$1/2$]             (xc)               {};
\node[node, label=$l_x$, label=right:$1/2$] (xr) [right=of xc] {};
\draw[-] (xc.east) -- (xr.west);

\node[node,label=left:$b_{l_x}^c$] (bx) [below=of xr] {};

\node[node, label=left:${l_x}^c$,label=below:$1/2$] (lxc) [below=of bx]  {};
\node[node,label=above:$c'$, label=below:$q$]   (cp)  [right=of lxc] {};
\node[node,label=right:${l_y}^c$, label=below:$1/2$] (lyc) [right=of cp]  {};
\draw[-] (lxc.east) -- (cp.west);
\node[rectangle, minimum width=1cm, text width=.45\linewidth][below=of cp] {\footnotesize\mar{(b) Clause $P_3$ cut on the right-hand side}};
\draw[->] (xr.south) -- (bx.north);
\draw[->] (lxc.north) -- (bx.south);

\node[node,label=right:$b_{l_y}^c$] (by) [above=of lyc] {};
\draw[->] (lyc.north) -- (by.south);

\node[node, label=$l_y$, label=right:$1/2$] (yl)  [above=of by] {};

\draw[->] (yl.south) -- (by.north);
\draw[->] (lyc.north) -- (by.south);

\end{tikzpicture}
\caption{Clause $c=(l_x,l_y)$, $l_x=$ False, $l_y=$ True
}
\label{fig:clausevf}
\end{figure}

\begin{figure}[tb]
\centering
\begin{tikzpicture}
\node[node, label=$l_x$, label=left:$1/2$] (xr) {};

\node[node,label=left:$b_{l_x}^c$] (bx) [below=of xr] {};

\node[node,label=left:${l_x}^c$, label=below:$1/2$] (lxc) [below=of bx] {};
\node[node,label=above:$c'$, label=below:$q$]   (cp)  [right=of lxc]      {};
\node[node,label=right:${l_y}^c$, label=below:$1/2$] (lyc) [right=of cp]       {};
\draw[-] (cp.east) -- (lyc.west);

\draw[->] (xr.south) -- (bx.north);
\draw[->] (lxc.north) -- (bx.south);

\node[node,label=right:$b_{l_y}^c$] (by) [above=of lyc] {};
\draw[->] (lyc.north) -- (by.south);

\node[node, label=$l_y$, label=right:$1/2$]  (yl)  [above=of by] {};

\draw[->] (yl.south) -- (by.north);
\draw[->] (lyc.north) -- (by.south);

\end{tikzpicture}
\caption{Clause $c=(l_x,l_y)$, $l_x=$ True, $l_y=$ True.  \mar{Example where the clause $P_3$ is cut on the left-hand side.}}
\label{fig:clausevv}
\end{figure}
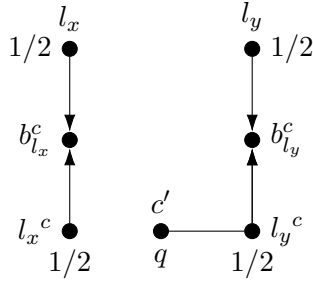


When $l_x$ and $l_y$ are both False (see Figure~\ref{fig:clauseff}), the risk for the binding vertex on the left is $1-(1-1/2)(1-1/2)(1-1/2)=7/8$, whereas, for the vertex on the right it is $1-(1-1/2)(1-1/2)(1-1/2)(1-q)=1-1/8+q/8=7/8+q/8$. In total, we have $7/4+q/8$. Then: 
\begin{equation}\label{eq:rho0}
    \rho_0=(7/4+q/8)\cdot K_0
\end{equation}

When $l_x$ and $l_y$ are both True (see Figure~\ref{fig:clausevv}), the risk for the binding vertex on the left is $1-(1-1/2)(1-1/2)=3/4$, whereas, for the vertex on the right it is $1-(1-1/2)(1-1/2)(1-q)=1-1/4+q/4=3/4+q/4$. In total, we have $3/2+q/4$. Then: 

\begin{equation}\label{eq:rho2}
   \rho_2=(3/2+q/4)\cdot K_2
\end{equation}

When $l_x$ is False and $l_y$ is True, we have two cases, as depicted in Figure~\ref{fig:clausevf}-a and Figure~\ref{fig:clausevf}-b. In the first case \ales{(Figure~\ref{fig:clausevf}-a)}: the risk for the binding vertex on the left is  $1-(1-1/2)(1-1/2)(1-1/2)=7/8$, whereas, for the vertex on the right it is $1-(1-1/2)(1-1/2)(1-q)=1-1/4+q/4=3/4+q/4$, and in total we have $13/8+q/4$. In the second case \ales{(Figure~\ref{fig:clausevf}-b)}: the risk for the binding vertex on the left is  $1-(1-1/2)(1-1/2)(1-1/2)(1-q)=1-1/8+q/8=7/8+q/8$, whereas, for the vertex on the right it is $1-(1-1/2)(1-1/2)=3/4$, and in total
we have $13/8+q/8$. Then, by choosing opportunely \ales{the cut that minimizes  these two values on the $P_3$ path representing the clause $(l_x,l_y)$}, we have: 

\begin{equation}\label{eq:rho1}
   \rho_1= (13/8+q/8)\cdot K_1
\end{equation}

\mar{Since} $K_0=m-(K_1+K_2)$, 
\mar{we deduce from Relations~\ref{eq:rho0}, \ref{eq:rho2}, and \ref{eq:rho1}:} 

\begin{equation}\label{eq:rhomin}
  \begin{array}{rl}
    \rho=\rho_0+\rho_1+\rho_2&= m(\frac 7 4 + \frac q 8) - (K_1+K_2)(\frac 7 4 + \frac q 8)+ (\frac{13} 8 + \frac q 8) K_1 + (\frac 3 2 + \frac q 4)K_2\\
    &= m(\frac 7 4 + \frac q 8) - \frac {K_1} 8 +(- \frac 1 4 + \frac q 8)K_2.
  \end{array}
\end{equation}

\mar{Since} $q \leq 1$, $\rho \leq m(\frac 7 4 + \frac q 8)- \frac {(K_1+K_2)} 8 \leq m(\frac 7 4 + \frac q 8)- \frac {K} 8$.

Then, by recalling that the total induced risk on variable vertices is $2n\nu$ and \mar{the one} on clause vertices is $m\omega(\frac 3 2 + q)$, we have a total risk $$\rho(\Gamma^I_H)= 2n\nu+ m\omega(\frac 3 2 + q) +\rho \leq 2n\nu+ m\omega(\frac 3 2 + q)+ m(\frac 7 4 + \frac q 8)- \frac {K} 8= R.$$

\bigskip
Conversely, let us assume that the instance $(\Gamma^I,\mathds{1},\pi_i,\kappa,\varphi,B,R)$ of \gab{\MSFS}
admits a cut system $H$ for $\Gamma^I$ such that $\kappa(H)\leq B$ and $\rho(\Gamma^I_H)\leq R$. Then we prove that the instance $I =(X,C,K)$ of \RSP, has an assignment for the variables in $X$ such that at least $K$ clauses are satisfied.


As 
\mar{already noticed, $H$ has necessarily} one cut in each $P_3$ 
representing variables in $X$ and one cut in each $P_3$ 
representing clauses in $C$. The induced risk of the vertices in variable $P_3$s is $2n\nu$, 
\mar{while} the induced risk of the vertices in clause $P_3$s is $m\omega(\frac 3 2 + q)$. 
\mar{Since} the total risk is less than $R$, the \mar{risk $\rho'$} induced 
on the binding vertices \mar{satisfies} 

\begin{equation}\label{eq:rho'sup}
\rho'\leq m\left(\frac 7 4 + \frac q 8\right) - \frac {K} 8. 
\end{equation}

By using the same notations  $K_0, K_1,K_2$ \mar{as above and since Relation~\ref{eq:rhomin} gives the minimum possible risk  for fixed $K_0,K_1,K_2$,we deduce} 

\begin{equation}\label{eq:rho'inf}
  \begin{array}{rl}
   \rho' &\geq  m(\frac 7 4 + \frac q 8) - \frac {K_1} 8 +(- \frac 2 8 + \frac q 8)K_2\\
   &\geq m(\frac 7 4 + \frac q 8) -\frac {K1+(2-q)K2} 8\\ &\geq m(\frac 7 4 + \frac q 8) -(2-q)\frac {K1+K2} 8 .
\end{array}
\end{equation}

\mar{We deduce from Relations~\ref{eq:rho'sup} and~\ref{eq:rho'inf}}
that $m(\frac 7 4 + \frac q 8) -\frac {K1+K2} 8 (2-q)\leq m(\frac 7 4 + \frac q 8)- \frac {K} 8$, which implies:

\begin{equation}\label{eq:K1+K2}
K_1+K_2\geq \frac K {2-q}.
\end{equation}



\mar{We chose: $q = 1 - \frac 1 {2K-1}$. This implies:}
$$q= 2 -\frac K {K - \frac 1 2}\Leftrightarrow \frac K {2-q}= K - \frac 1 2$$

\mar{Then, Relation~\ref{eq:K1+K2} implies $K_1+K_2\geq  K - \frac 1 2$. Since $K, K_1,K_2$ are integers, this  implies $K_1+K_2\geq  K$ and thus, $I$ is a positive instance.} 

\mar{Note finally that the edge costs in the reduction are 1 or $s$ and the vertex values  are $1, \nu$ or $\omega$. All these values are integral and polynomially bounded with respect to $n$ and $m$ and thus, to the size of the \MSFS instance. Indeed, $s=m+1$, $\omega= \frac {8m} {1-q}$ and with $q = 1 - \frac 1 {2K-1}$, we get $\omega=8m(2K-1)$. Note that $K\leq m$. Then,  $\nu=8m\left(\frac 5 2 \omega + 2\right)=20m\omega +16m$. This concludes the proof. } 

\end{proof}


\mar{The proof of Proposition~\ref{th:complete-planar} is written using different values on vertices to keep it as simple as we could. 
In what follows, we make a simple remark that if values are integral and polynomially bounded, then  a simple transformation reduces the problem to the case where all vertices have the same value~1.}

\begin{prop}\label{prop:value-reduction}
 
  \mar{For any graph class ${\cal C}$ closed under subdivision of edges, \MSFS with polynomially bounded vertex values and for graphs  in ${\cal C}$ polynomially reduces to its particular case where all vertices have value~1. The transformation preserves the maximum degree. }
\end{prop}

\begin{proof}
\mar{Consider an instance $I=(\Gamma,\mathds{1},\pi_i,\kappa,\varphi,B,R)$ of \MSFS, where $\Gamma=(V,E)\in {\cal C}$ and we assume that there is a polynomial $P$ such that $\forall v\in V, \varphi(v)\leq P(|V|)$. We build in polynomial time an instance $I'=(\Gamma',\mathds{1},\pi'_i,\kappa',\varphi_\mathds{1},B,R)$
with $\Gamma'=(V',E')\in {\cal C}$, $\varphi_\mathds{1}$ is the  constant function on $V'$ that maps any vertex to 1 and $I'$ is positive if and only if $I$ is positive. The budget and the risk threshold remain unchanged.}

\mar{For every vertex $v\in V$ of degree~$d$, we denote $u_1,\ldots,u_d$ the neighbors of $v$ and  insert $\mu_i\geq 0$ vertices on the edge $vu_i$, $i=1,\ldots,d$ such that $1+\sum\limits_{i=1,\ldots,d}\mu_i=\varphi(v)$. We denote $X(v)$ the set of new vertices; all edges between two vertices in $\{v\}\cup X(v)$ are non-directed and  have the same cost $B+1$. Vertices in $X(v)$ have an ignition probability~0 while the ignition probability of $v$ is unchanged: $\pi'_i(v)=\pi_i(v)$. The edge $vu_i$ is replaced by the edge $z_i^vu_i$, where $z_i^v$ is the  vertex inserted on $vu_i$ that is linked to $u_i$ (could be $v$ if $\mu_i=0$). If  $vu_i$ is directed, then 
$z_i^vu_i$ is directed with the same orientation. If we perform this transformation for every vertex, then we get an instance $I'$ in a graph $\Gamma'=(V',E')\in {\cal C}$, where all vertices have the same value~1. We can see $V$ and $E$ as subsets of $V'$ and $E'$, respectively and there is a one to one correspondence between edges of cost at most $B$ in $\Gamma$ and  in $\Gamma'=(V',E')$ and the cost is preserved by this correspondence. In particular,  any cut system $H\subset E$ in 
$\Gamma$ such that $\kappa(H)\leq B$ can be seen as a cut system in 
$\Gamma'$ with the same cost. It is straightforward that $\rho(\Gamma_H)=\rho(\Gamma'_H)$. Indeed,  if a vertex $v$ burns with some probability in $\Gamma_H$, then, in $\Gamma'_H$ the $\varphi(v)$ vertices in $\{v\}\cup X(v)$ will burn with the same probability. Since the transformation can be performed in polynomial time, the proof is complete.}
\end{proof}

\mar{The transformation in Proposition~\ref{prop:value-reduction} preserves planarity and the maximum degree. However, it does not necessarily  preserve bipartite graphs. Note however that bipartiteness can easily be imposed. First, we notice that multiplying all vertex values by the same number just induces multiplying the risk of any solution by this constant; so, it does not change the problem. Then, we propose a first transformation ensuring that all vertices have an odd degree. First multiply all vertex values as well as the risk threshold  by 2 to ensure all values are at least 2. If a vertex $v$ has an even degree, then add a pending vertex of value~1 and probability of ignition 0, reduce the value of $v$ by 1 and define the cost of the related edge as $B+1$. This defines an equivalent instance where all vertices have an odd degree. In addition if the maximum degree is odd, this operation does not change it, else it adds 1 to the maximum degree. Denote now by $\Delta$ the maximum degree in the new graph and multiply  all vertex values as well as the risk threshold by $2\Delta$ to obtain an equivalent instance where  each vertex has an even value greater than its degree. In this case, we can always perform the transformation in Proposition~\ref{prop:value-reduction} ensuring that all $\mu_i$s are odd. For instance, chose all $\mu_i$s but one equal to~1.  This ensures the new graph $\Gamma'$ is bipartite. Note finally that all edges in $\Gamma$ are preserved with their cost and the new edges have the cost $B+1$. Without loss of generality we can assume $B\leq \sum\limits_{e\in E}\kappa(e)$ and consequently, if all edge costs are polynomially bounded in $\Gamma$, so are they in $\Gamma'$.
Using this argument we deduce the following corollary:}

\begin{coro}\label{cor:samerisk}
\mar{\MSFS remains NP-complete in bipartite planar \gab{graphs} of maximum degree~5 when all the vertices' values are 1 and edge costs are polynomially bounded.}
\end{coro}

A natural question is whether we can add the constraint that edge costs are all~1. A first answer is that replacing an edge of cost~$c$ by  $c$ parallel edges, each of cost~1,  makes the problem equivalent. The transformation preserves planarity and bipartiteness but it does not preserve low degree. In what follows, we sketch a polynomial reduction that allows to maintain the degree bounded.


\begin{prop}\label{prop:cost-reduction}
  \mar{\MSFS with polynomially bounded edge costs, all vertex values~1  and a rational probability system with a polynomial least common multiple, polynomially reduces to \MSFS with all edge costs and vertex values equal to~1 in a graph of maximum degree~4.}
\end{prop}

\begin{proof}

Consider an instance $I=(\Gamma,\mathds{1},\pi_i,\kappa,\mathds{1},B,R)$ of \MSFS
with all vertex values equal to~1. We assume $\Gamma=(V,E)$ with $n=|V|$, edge costs are integers and probabilities are all rational. We also assume there is a polynomial integral function  
$f$ such that $\forall e\in E, \kappa(e)\leq f(n)$ and $\forall x\in V, f(n)\pi_i(x)\in \mathbb{N}$. This last relation ensures that for all cut system $H$ in $\Gamma$, $f(n)\rho(\Gamma_H)\in \mathbb{N}$. 
\mar{To simplify further expressions, we define $C=\left\lceil \frac{B}{2}\right\rceil$}.
The construction depends on a value $M$, polynomially bounded, chosen as follows:

\mar{
\begin{equation}\label{eq:M}
 M=\max\left(1+C\left\lceil \sqrt{2Rf(n)+1}\right\rceil,|E|f(n)\right)
\end{equation}
} 
 
We build an instance $I'=(\Gamma',\mathds{1},\pi'_i,\mathds{1},\mathds{1},B',R')$
with all edge cost equal to~1 such that $I'$ is positive if and only if $I$ is positive. In addition $\Gamma'$ has maximum degree~4. $\Gamma'$ is obtained from $\Gamma$ by replacing each vertex $x$ with a $(M\times M)$ non-directed square grid $Q_x$ with $M^2$ vertices. Edges of $Q_x, x\in V$ are called {\em $Q$-edges} in $\Gamma'$. Each edge $(x,y)$ of cost $\kappa((x,y))$ is replaced with a set $J_{(x,y)}$ of $\kappa((x,y))$ edges, each of cost~1; such edges will be called {\em joining edges} in $\Gamma'$. All edges incident to $x$ in $\Gamma$ correspond, in $\Gamma'$, to $\sum\limits_{(x,y)\in E}\kappa((x,y))$ joining edges incident to the perimeter of $Q_x$ and with extremities equally spread along this perimeter. 
\mar{Since $M\geq |E|f(n)\geq \sum\limits_{e\in E}\kappa(e)$,}  the perimeter of $Q_x$ is long enough. All vertices in 
$Q_x$ have the same ignition probability of~ $1-(1-\pi_x)^{\frac{1}{M^2}}$, where $0^{\frac{1}{M^2}}=0$.
Finally, we define $B'=B$ and $R'=M^2R$. Note that the maximum degree of $\Gamma'$ is~4 and that the construction can be performed in polynomial time.

Assume first that $I$ has a cut system $H\subset E$ such that $\kappa(H)\leq B$ and $\rho(G_H)\leq R$. We then define a cut system $H'=\bigcup\limits_{(x,y)\in H}J_{(x,y)}$ in $\Gamma'$. With the edge costs in $\Gamma$ and $\Gamma'$ we have $\kappa'(H')=\kappa(H)\leq B$.  It is straightforward to verify that $\rho(\Gamma'_{H'})= M^2\rho(\Gamma_{H})$ since $ Q_x$'s probability of burning in $\Gamma'_{H'}$ is exactly the probability that $x$ burns in $\Gamma_{H}$ and $\varphi'(Q_x)=M^2\varphi(x)$. So, $I'$ is positive.

Conversely, assume $I'$ is positive and let $H'$ be a cut system satisfying $\kappa'(H')\leq B$ and $\rho(\Gamma'_{H'})\leq M^2R$.  

For a vertex $x\in V$, a {\em connected component} of $\Gamma'_{H'}[Q_x]$ is called {\em small} if its size is at most $C^2$ and it is {\em large} instead. Relation~\ref{eq:M} implies in particular $M>B$. We establish few claims:

\begin{claim}\label{claim:Qx}
$\forall x\in V$, $\Gamma'_{H'}[Q_x]$ has exactly one large component.
\end{claim}

\begin{proof}
Since all edge costs in $I'$ are~1, we have $|H'|\leq B$.  Since $M>B$, removing $B$ edges from the $M\times M$ grid $Q_x$ allows to disconnect at most $C^2$ vertices from the rest of the grid (this maximum is obtained if removed edges disconnect a corner $C \times C$ of the grid $Q_x$). \mar{Relation~\ref{eq:M} implies $M^2 - C^2 > C^2$ and consequently the remaining vertices in $Q_x$ constitute a large component.} This concludes the proof of the claim. 
\end{proof}

The same argument allows to show: 

\begin{claim}\label{claim:small}
The total size of all small components in $\Gamma'_{H'}$ is at most $C^2$.
\end{claim}

To derive from $H'$ a cut system in $\Gamma$, we first transform $H'$ into $H''$ that only includes  joining edges:

\begin{itemize}
    \item Any joining edge in $H'$ is added  to $H''$;
    \item Any joining edge adjacent to a small component of  $\Gamma'_{H'}[Q_x]$ for some $x\in V$ is added to $H''$;
\end{itemize}

Note that a similar argument as in the previous claim shows that   the number of $Q$-edges in $H'$ is at least equal to the number of joining edges adjacent to a small component of  $\Gamma'_{H'}[Q_x]$. As a consequence, $|H''|\leq |H'|$. 

We denote $V_Q$ the set of vertices of small components in $\Gamma'_{H'}$ and consider the graph $\Gamma''=\Gamma'[V\setminus V_Q]$. 

\begin{claim}\label{claim:rhoH''}
$\rho(\Gamma''_{H''})\leq \rho(\Gamma'_{H'})$.
\end{claim}
\begin{proof}
Connected components of $\Gamma''_{H''}$ are contained in connected components of $\Gamma'_{H'}$.
\end{proof}

We now define a cut system $H$ in $\Gamma$ as follows: for every edge $e\in E$, add it in $H$ if and only if $J_e\subset H''$. It is straightforward to show that $\kappa(H)\leq |H''|$.

\mar{
\begin{claim}\label{claim:rhoH}
$\rho(\Gamma_H)\leq 
\frac{\rho(\Gamma''_{H''})}{M^2-2C^2}$
\end{claim}
}

\begin{proof}

Since all vertices have value~1, we deduce from Lemma~\ref{lem:px} that $\rho(\Gamma_H)= \sum\limits_{x\in G_H} p_x$. Consider a vertex $x\in V$ and the related vertex set $Q_x\setminus V_Q$ in  $\Gamma''$.  All vertices in $Q_x\setminus V_Q$ are connected in $\Gamma''_{H''}$ due to Claim~\ref{claim:Qx}. 

Consider an edge $(y,x)$ in $\Gamma_H$, by definition of $H$, there is at least one edge from  $Q_y\setminus V_Q$ to $Q_x\setminus V_Q$. So, 
let  $z_x\in Q_x\setminus V_Q$, $y\in U_{x,H}$ in $\Gamma$ and $z_y\in Q_y\setminus V_Q$, we have $z_y\in U_{z_x,H''}$ in $\Gamma''$. Conversely, if $z_y\in U_{z_x,H''}$ in $\Gamma''$ with $z_x\in Q_x\setminus V_Q$ and $z_y\in Q_y\setminus V_Q$, then $y\in U_{x,H}$ in $\Gamma$. 

If we call $p''_{z_x}$ the probability that $z_x\in Q_x\setminus V_Q$ burns in $\Gamma''_{H''}$, we have: 

\begin{equation*}
\begin{array}{rl}
    1-p''_{z_x}&=\prod\limits_{t\in U_{z_x,H''}}(1-\pi_i(t))\\
    &\leq \prod\limits_{y\in U_{x,H}}(1-\pi_i(y))^{\frac{M^2-C^2}{M^2}}\\
    &=(1-p_x)^{\frac{M^2-C^2}{M^2}}
    \end{array}
\end{equation*}


\mar{If $p_x<1$, we consider the real function $h: z\mapsto (1-p_x)^z=e^{z\log(1-p_x)}$. It is decreasing and convex. As a consequence, on the interval $[0,1]$, $h$ is bounded above by the linear function equal to $h(0)=1$ for $z=0$ and equal to $h(1)= (1-p_x)$ for $z=1$. So, $\forall 0\leq z\leq 1$,  $(1-p_x)^z\leq 1-zp_x$. If $p_x=1$, then the inequality also holds.}

\mar{Since, for any vertex $x\in V$ in $\Gamma$, the large component associated with $x$ in $\Gamma''$ includes at least $M^2-C^2$ vertices, we  have:}

\mar{
\begin{equation*}
\begin{array}{rl}
\rho(\Gamma''_{H''})&\geq (M^2-C^2)\sum\limits_{x\in V}\left(1-(1-p_x)^{\frac{M^2-C^2}{M^2}}\right)\\
&\geq (M^2-C^2)\sum\limits_{x\in V}\left((1-\frac{C^2}{M^2})p_x\right)\\
&\geq (M^2-C^2)(1-\frac{C^2}{M^2})\rho(\Gamma_H)\\
&\geq (M^2-2C^2)\rho(\Gamma_H)
 \end{array}
\end{equation*}
}

We deduce:

\mar{
\begin{equation*}
  \rho(\Gamma_H) \leq \frac{\rho(\Gamma''_{H''})}{M^2-2C^2}
\end{equation*}
}

which completes the proof of the claim.
\end{proof}

Using  Claim~\ref{claim:rhoH''}, Claim~\ref{claim:rhoH} and  $\rho(\Gamma'_{H'})\leq M^2R$ we deduce:

\mar{
\begin{equation}\label{eq:rhoH}
\begin{array}{rl}
  \rho(\Gamma_H)&\leq \frac{M^2R}{M^2-2C^2} \\
  &\leq  R+ \frac{2C^2R}{M^2-2C^2}
 \end{array}
\end{equation}
}

Relation~\ref{eq:M} implies 
\mar{$\frac{2C^2R}{M^2-2C^2}<\frac{1}{f(n)}$}. So, Relation~\ref{eq:rhoH} implies that $\rho(\Gamma_H)< R+\frac{1}{f(n)}$. Since $f(n)\rho(\Gamma_H)\in \mathbb{N}$, we deduce $\rho(\Gamma_H)\leq R$, and consequently $I$ is positive, which completes the proof of Proposition~\ref{prop:cost-reduction}.
\end{proof}

The transformation in Proposition~\ref{prop:cost-reduction} preserves planarity. We can easily ensure it preserves bipartiteness. \newmarc{Assume indeed that the  original graph $\Gamma=(V,E)$ is bipartite with two parts  black and white and consider for instance a black vertex $x\in V$. We then consider any black and white partition of the grid $Q_x$ and branch joining edges along the perimeter of  $Q_x$ only on black vertices. Doing it for any $x$ will ensure that $\Gamma'$ is bipartite}. 
Since \mar{$M\geq \sum\limits_{e\in E}\kappa(e)$}, the perimeter is long enough.

In the proof of Proposition~\ref{th:complete-planar}, the reduction involves a class of \MSFS instances with only three different ignition probabilities, 0, $\frac{1}{2}$ and 
$q = 1 - \frac 1 {2K-1}$, where $K$ can be chosen not greater than the number of clauses which is less than the number of vertices in the instance of \MSFS. \mar{Edge costs are either 1, $m+1$ or $(n+1)(m+1)$, where both $n$ and $m$ are less than the number of vertices in the  instance of \MSFS.   So, we can choose $f(n)=2(2K-1)(n+1)(m+1)$ that satisfies all requirements of Proposition~\ref{prop:cost-reduction}. Applying successively the reductions in   Proposition~\ref{th:complete-planar}, Corollary~\ref{cor:samerisk} and Proposition~\ref{prop:cost-reduction} from an instance $(\Phi,K)$ of \RSP allows to prove the following Theorem that constitutes the main result of this section: } 

\begin{thm}\label{theo:bipartite-planar4}
\mar{\MSFS is NP-complete in bipartite planar of maximum degree~4 and  all vertex values and edges costs are 1.}
\end{thm}

Note that, in Proposition~\ref{prop:cost-reduction} the probability system $\pi_i'$ 
\newgab{for the \MSFS with all edge costs and vertex values equal to 1, in a graph of maximum degree 4,}
is not necessarily rational neither polynomially bounded. 
We let as an  open problem the complexity of \MSFS in bipartite planar graphs with all vertex values and edges costs equal to 1 and rational and polynomially bounded ignition probabilities. 

\newmar{We conclude this section with a remark about the complexity in grid graphs that are natural when considering an homogeneous landscape divided into regular square-like areas. A planar graph of maximum degree~4 can be embedded in polynomial time into a grid~\cite{embed} in such a way that vertices map to vertices in the grid and edges map to non-crossing paths in the grid. We can also easily ensure that the embedding is a subgraph of the grid, also called a {\em subgrid}. In other words, there is a subdivision of the original graph that is a subgrid and the construction can be performed in polynomial time. If the new vertices (produced in the subdivision) are all of value 0 and all edges have the same cost~1, then \MSFS in this subgrid is equivalent to \MSFS in the original graph. Now, if we consider any grid that contains the subgrid as a subgraph, assign the value~1 and the ignition probability $\pi_i=0$ to the added vertices and the cost~0 to the added edges, then \MSFS in this grid is equivalent to \MSFS in the subgrid since we can cut any edge of value~0 without changing the total cost. As a result, we immediately deduce: }

\begin{prop}\label{pro:grid}
\MSFS is NP-complete if:
  \begin{itemize}
      \item the graph is a subgrid with binary vertex values and unitary edge costs;
      \item the graph is a grid with binary vertex values and edge costs. 
  \end{itemize}
\end{prop}

The case with binary ignition probabilities would be as well an interesting open case. In the next section, we identify a polynomial problem for this case.

\section{\MSFS in Trees}
\label{sec:polynomially}

\mar{The hardness results in the previous section motivate the question of identifying some polynomial cases for \MFS. Since 
\MSFS revealed to be hard in restricted cases and since the complexity of 
\MSFS with binary ignition probabilities is still open, it seemed to us relevant to start with this cases. 
We propose a polynomial time algorithm in trees with vertex values and edge cost equal to~1 and ignition probabilities are binary.}


As seen in Proposition~\ref{pro:partition}, \MSFS is hard on trees (even on stars) if edge costs and vertex values can be any integer. 

In this section, we present a polynomial algorithm that solves \mar{\MSFS} with all edge costs and vertex values equal to~1 
in a general tree with a subset of burning vertices. 
The algorithm outputs the maximum number of savable vertices and the corresponding cut system.

Given a tree $T=(V,E)$, we consider an instance $(T,\pi_s,\pi_i,\kappa,\varphi,B)$ of the \MFS problem \mar{(optimization version)} where:
\begin{itemize}
    \item $\pi_s(e)=1$, for each $e\in E$ \mar{(this is an instance of \MSFS);}
    \item $\pi_i(v)=1$, for $v \in V' \subseteq V$;
    and $\pi_i(v)=0$ for $ v \in V \setminus V'$;
    \item $\kappa(e) = 1$, for each $e\in E$;
    \item $\varphi(v)=1$, for each $v\in V$;
    \item a given budget $B$. 
\end{itemize}

We devise a polynomial time algorithm that computes a {\em cut system} $H\subset E$ such that $\kappa(H)\leq B$ 
 minimizing the risk, which is equivalent to 
maximizing the number of saved vertices.

Given an instance $(T,\pi_s=\mathds{1},\pi_i,\kappa,\varphi,B)$ of \mar{the optimization version of} \MSFS, where $T$ is a tree,
we choose a vertex as the root, \mar{we orient edges from the root to the leaves}, for every vertex, we define an order of its children  and then we number the vertices $v_0, \ldots,v_{|V|-1}$ in post order. 
Let $T_i$ be the subtree rooted in the vertex $v_i$ \mar{that includes only vertices that are descendant of $v_i$ in $T$}. By property of the post order, if $T_i$ is a subtree of $T_j$, then $i < j$. Given a cut system $H$, we denote with $\pi_o(v)$  the probability, in $G_H$, that a vertex burns in the final setting. $\pi_o(v)=0$ if the vertex $v$ does not burn in the solution or $\pi_o(v)=1$ if the vertex $v$ burns in the final setting.

\textbf{Input}: A tree $T$, a set of vertices $V' \subset V$  such that $\pi_i(v)=1$ for $v\in V'$, $\pi_i(v)=0$ for $v\notin V'$ and a budget $B$.

\textbf{Output}: A {\em cut system} $H\subset E$ 
of cost at most the given budget $B$ and maximizing the number of vertices $v$ such that $\pi_o(v)=0$. 

%
%

\subsection{Algorithm Description}\label{subsec:algo_descript}

The algorithm 
\mar{computes an} optimal solution \mar{using a double dynamic programming process}. 
\mar{The main dynamic programming process} computes an optimal solution 
for each subtree starting from the leaves and following a post-order visit. \mar{At each step, the algorithm} computes an optimal solution for a subtree $T_i$ \mar{using solutions for the subtrees induced by the children of $v_i$} 
(already computed due to the post-order visit), \mar{using a second dynamic programming process}. The procedure continues until 
an optimal solution is computed for the whole tree.

In a more formal way, given the tree $T=(V,E)$ and a vertex $r$ as the root, the algorithm visits and numbers the vertices in post-order from $v_0$ to $v_{|V|-1}$ ($v_{|V|-1} \equiv r$). 
Given a vertex $v_i$, we denote with $v_{i_j}$ the $j$-th children of $v_i$ (in the chosen ordering of children of each vertex) and, by \mar{definition} of the post-order, if $j < j'$, then  $i_j < i_{j'}<i$ .


The algorithm then builds two tables (see Figure~\ref{fig:treeAlgov2}). For both tables, rows and columns are numbered starting from zero. 

~

{\underline{\mar{Table $A$:}}}

Table $A$ (\mar{main dynamic programming process}) has $|V|$ rows and $B+1$ columns. It contains, for each row $i$ and each budget $b \in 
\mar{\{0, \ldots, B\}}$, 
the number of vertices of $T_i$ that can be saved \mar{for} two different scenarios \-- if $v_i$ burns and 
if $v_i$ does not burn \-- in the final setting, and, for each case, a  
corresponding cut system of $b$ edges in $T_i$. If $v_i\in V'$, then only the case where $v_i$ burns is taken into account. 
\pier{\mar{So}, every 
\mar{entry} in the table is a 
\mar{4-tuple} $A_{i,b}=(f^+,f^-,H^+,H^-)$ related to the subtree $T_i$, rooted in $v_i$ \mar{and for a} given budget $b$. 
$f^+$ is the optimal 
\mar{value} when vertex $v_i$ burns (i.e., $\pi_o(v_i)=1$), 
$f^-$ is the optimal 
\mar{value} when vertex $v_i$ is not burning in the final setting (i.e., $\pi_o(v_i)=0$), 
$H^+$ and $H^-$ are optimal cut systems associated to $f^+$ and $f^-$, respectively.} \mar{In what follows, we will respectively denote $A_{i,b,f^+}$, $A_{i,b,f^-}$, $A_{i,b,H^+}$, and $A_{i,b,H^-}$ the four components of $A_{i,b}$.}

\begin{algorithm}[t]
\caption{Optimal Tree Cut}
\label{treecut}
\hspace*{\algorithmicindent} {\bf Input:} instance $I=(T,\mathds{1},\pi_i,\kappa,\varphi,B)$ of \MSFS with $\pi_i(v) \in\{1,0\}$.\\
\hspace*{\algorithmicindent} {\bf Output:} (optimal cut system $H$, the number of saved vertices)
\begin{algorithmic}[1]
\Procedure{TableA}{$I$}
\State pick a vertex $r$ as the root \label{algo1:root}
\State let $v_0, v_1, \ldots, r = v_{|V|-1}$ the vertices of $T$ visited in post-order from $r$ \label{algo1:postorder}
\For{$i = 0, 1, \ldots, |V|-1$}  \label{algo1:tablebstart}
\State $ A_i\gets$  \sc{TableST}($T$,$A$,$v_i$,$B$)
\EndFor  \label{algo1:tablebend}
\If{$A_{|V|-1,B,f^-} > A_{|V|-1,B,f^+}$} \label{algo1:maxstart}
\State \textbf{return}  $(A_{|V|-1,B,H^-}, A_{|V|-1,B,f^-})$
\Else
\State  \textbf{return} $(A_{|V|-1,B,H^+}, A_{|V|-1,B,f^+})$
\EndIf \label{algo1:maxend}
\EndProcedure
\end{algorithmic}
\end{algorithm}

Algorithm \ref{treecut} builds Table A. 
It is straightforward: a root is chosen at line~\ref{algo1:root}, a post order visit is executed at line~\ref{algo1:postorder}, each row of Table A is filled using procedure \pier{TableST} (see Algorithm \ref{tableB} described below) at lines~\ref{algo1:tablebstart}--\ref{algo1:tablebend}. 
Once the values for Table A are computed for all the subtrees $T_i$, the element in the last row and column contains an optimal solution for both possible states of the root $r$ (i.e., burns or does not burn). 

An optimal solution for the problem is then the maximum between these two values, with the corresponding cut system. It is returned at lines~\ref{algo1:maxstart}--\ref{algo1:maxend}.

{\underline{\mar{Table ST:}}}\\
Table \mar{ST (auxiliary dynamic programming process)}  is built for each subtree $T_i$ rooted in $v_i$ in order to compute the $i$-th row of Table A using the row $i-1$. Each row of Table \pier{ST} stores, for each possible budget $b \in 
\mar{\{0,\ldots, B\}}$, solutions (value and cut system)  for some subtrees of $T_i$ if $v_{i}$ burns and, when $v_{i}\notin V'$, if it does not burn. Table \pier{ST}  has $k + 1$ rows, where $k$ is the number of children of $v_i$, and $B+1$ columns. The first row 
stores 
solutions (optimal value and a related cut system) for the root $v_i$ without descendants (so, a single vertex). If $k>0$, 
\pier{then} each row $j > 0$ is filled with solutions 
for the subtree $T_i^j$ obtained by connecting to $v_{i}$ 
the subtrees $T_{i_1}, \ldots, T_{i_j}$. The last row then corresponds to $T_i$. Each column $b$ of Table \pier{ST} corresponds to the budget used for the related subtree. Note that the last row of Table \pier{ST} for vertex $i$ is the $i$-th row of Table A. 
\pier{Like in Table A, every entry in Table \pier{ST} is a 
\mar{4-tuple} $ST_{j,b}=(f^+,f^-,H^+,H^-)$, for subtree $T^j_i$ and budget $b$.} \mar{We use the same abbreviated notations as for Table A, respectively denoting $ST_{j,b,f^+}$, $ST_{j,b,f^-}$, $ST_{j,b,H^+}$, and $ST_{j,b,H^-}$ the four components of $ST_{j,b}$.}


\begin{algorithm}[t]
\hspace*{\algorithmicindent} {\bf Input:} Tree $T$, table $A$, vertex $v_i$ and budget $B$\\
\hspace*{\algorithmicindent} {\bf Output:} optimal solutions for the subtree $T_i$, for each budget $b$ up to $B$ 
\caption{Subprocedure for table ST}
\label{tableB}
\begin{algorithmic}[1]   
\Procedure{TableST}{$T$,$A$,$v_i$,$B$}
\State let $(v_{i_1},v_{i_2}, \ldots, v_{i_k}) $ be the children of vertex $v_i $
\If{$\pi_i(v_{i})$=1}\label{algo2:leafstart}
\State $ST_{0,b}\gets (0,-\infty,\emptyset,\emptyset)$~~~~~ $\forall\ b \in \{0,1,\ldots,B\}$
\Else
\State  $ST_{0,b}\gets (0,1,\emptyset,\emptyset) $~~~~~ $\forall\ b \in \{0,1,\ldots,B\}$ \label{algo2:initnofire}
\EndIf 
\label{algo2:leafend}


\For{$j = 1, 2, \ldots, k$}

\For{$b = 0, 1, \ldots, B$ } 
\State $z\gets \underset{x \in \{0,\ldots ,b\}}{\arg\max}  \{ST_{j-1,x,f^+} + A_{{i_j},b-x,f^+} \}$ \Comment {$v_{i_j}$ and $v_{i}$ both burn} \label{algo2:bbstart}
\State $ST_{j,b,f^+}\gets ST_{j-1,z,f^+} + A_{{i_j},b-z,f^+}$
\State $ST_{j,b,H^+}\gets ST_{j-1,z,H^+} \cup A_{{i_j},b-z,H^+}$ 
\label{algo2:bbend}

\If{$(\pi_i(v_{i_j})=0) \wedge (b \geq 1)$}
\Comment {$v_{i}$ burns, $v_{i_j}$ not, budget $\geq 1$} \label{algo2:bnbudget}
\State $z'\gets \underset{x \in \{0,\ldots ,b-1\}}{\arg \max} \{ST_{j-1,x,f^+} + A_{{i_j},b-1-x,f^-} \}$ \label{algo2:bnstart}
\State $m'\gets ST_{j-1,z',f^+} + A_{{i_j},b-1-z',f^-}$ \label{algo2:bnend}
\If{$m' \geq ST_{j,b,f^+}$} \label{algo2:bbeststart}
\State $ST_{j,b,f^+}\gets m'$
\State $ST_{j,b,H^+}\gets ST_{j-1,z',H^+} \cup A_{{i_j},b-1-z',H^-} \cup \{(v_{i},v_{i_j})\}$ \label{algo2:bnaddcutsystem}
\EndIf \label{algo2:bbestend} 

\EndIf 

\If{$\pi_i(v_{i})=0$}
\Comment {$v_{i}$ does not burn}

\If{$b \geq 1$} \Comment {$v_{i_j}$ burns, budget $\geq 1$}  \label{algo2:nbbudget}
\State $z\gets \underset{x \in (0,\ldots,b)}{\arg\max}\{ST_{j-1,x,f^-} + A_{{i_j},b-1-x,f^+} \}$ \label{algo2:nbstart}
\State $ST_{j,b,f^-}\gets ST_{j-1,z,f^-} + A_{{i_j},b-1-z,f^+}$
\State $ST_{j,b,H^-}\gets ST_{j-1,z,H^+} \cup A_{{i_j},b-1-z,H^+} \cup \{(v_{i},v_{i_j})\}$ \label{algo2:nbend}
\EndIf 
\If{$\pi_i(v_{i_j})=0$} \label{algo2:nn}
\Comment {$v_{i_j}$ does not burn}
\State $z'\gets \underset{x \in (0,\ldots,b)}{\arg\max}\{ST_{j-1,x,f^-} + A_{{i_j},b-x,f^-} \}$  \label{algo2:nnstart}
\State $m'\gets ST_{j-1,z',f^-} + A_{{i_j},b-z',f^-}$ \label{algo2:nnend}
\If{$m' \geq ST_{j,b,f^-}$} \label{algo2:nnbeststart}
\State $ST_{j,b,f^-}\gets m'$
\State $ST_{j,b,H^-}\gets ST_{i-1,z',H^-} \cup A_{{i_j},b-z,H^-}$
\EndIf  \label{algo2:nnbestend} 
\EndIf  
\EndIf  
\EndFor
\EndFor
\State \textbf{return} $ST_k$  \label{algo2:lastrow}
\EndProcedure
\end{algorithmic}
\end{algorithm}

\clearpage

The values of \mar{row $j>0$} of Table \pier{ST} for subtree $T_i$ are computed as follows. 
The root $v_{i}$ of subtree $T_i$ can either burn or not if $v_{i}\notin V'$ and certainly burns if \pier{$v_{i} \in V'$}, and the same occurs for vertex $v_{i_j}$, root of $T_{i_j}$. 
So, when completing the $j$-th row of Table \pier{ST} for $T_i$, we have at most  four possible combinations. \mar{The case where a vertex does not burn is only considered if this vertex has a probability of ignition 0 (i.e., it is not in $V'$)}.

Two of these cases are concordant (i.e, \mar{$v_{i}$ and $v_{i_j}$ both burn or both do not burn}), and then no cut is needed between them, and  two of them are discordant (one of $v_{i}$ and $v_{i_j}$ is burning and one is not) 
and in this case the edge connecting them must be cut 
in any feasible solution. Note that $-\infty$ is used as value to state that there is no related feasible solution. This is the case 
when $v_i$ is in $V'$ and is stated not burning in
the combination under analysis.

More precisely, Algorithm \ref{tableB} implements the procedure for Table ST. 
For each budget $b \in 
\mar{\{0, \ldots, B\}}$, 
the first row is filled with values $(0,1, \emptyset, \emptyset)$ if $\pi_i(v_{i})=0$ (i.e., $v_{i}$ is not burning) and $(0,-\infty,\emptyset,\emptyset)$ if $\pi_i(v_{i})=1$ (i.e., $v_{i}$ is burning).

If $v_{i}$ is 
not a leaf, then the subsequent rows are computed considering the edge between vertex $v_{i}$ and the latest added vertex $v_{i_j}$.
To compute $ST_{j,b,f^+}$  
($v_{i}$ burns in the final setting), for budget $b$, we consider the latest added vertex $v_{i_j}$ 
\mar{and distinguishing the two possible cases} whether $v_{i_j}$ burns or not. 

%


\begin{enumerate}
    \item\label{item1} \emph{both $v_{i}$ and $v_{i_j}$ burn}.
The algorithm finds the maximum value obtained by allocating $x$ and $b-x$ budget on $T_i^{j-1}$ and $T_{i_j}$ for $x \in \{0 \ldots b\}$ and summing the values of $ST_{j-1,b,f^+}$ with $A_{{i_j},b,f^+}$ (lines~\ref{algo2:bbstart}--\ref{algo2:bbend}). 
    \item\label{item2} \emph{$v_{i}$ burns and $v_{i_j}$ does not burn}.
\mar{In any feasible solution}, we must cut the edge $v_{i}v_{i_j}$ and allocate the remaining $b-1$ budget on the two subtrees. So, we distribute $x$ and $b-x-1$ budget on the two subtrees and pick the maximum values between the sums of $ST_{j-1,b,f^+}$ with \mar{$A_{{i_j},b-1-x,f^-}$} (lines~\ref{algo2:bnstart}--\ref{algo2:bnend}). 
If the value computed in this case is greater than the value computed in the previous one, then the solution is updated and edge $v_{i}v_{i_j}$ is added to the cut-system $H^+$ (lines~\ref{algo2:bbeststart}--\ref{algo2:bbestend}).
\end{enumerate}

\mar{The best value obtained in cases~\ref{item1} and \ref{item2} corresponds to the correct value for  $ST_{j,b,f^+}$ assuming that values of  $ST_{j-1,b,f^+}$, $A_{{i_j},b,f^+}$ and $A_{{i_j},b,f^-}$ are correct. }

To compute $ST_{j,b,f^-}$ now, the algorithm evaluates the two cases in which 
$v_{i}$ does not burn in the final setting and $v_{i_j}$ either burns or not. \mar{This case is similar to the previous one with the possible outputs computed at} 
lines~\ref{algo2:nbstart}--\ref{algo2:nbend} for the discordant case, lines~\ref{algo2:nn}--\ref{algo2:nnbestend} for the concordant one. 

\mar{When table ST is completed, the last row contains the values of the solution for subtree $T_i$, for all varying budgets (line~\ref{algo2:lastrow}).}

\begin{figure}[t]
\begin{center}
\scalebox{0.65}{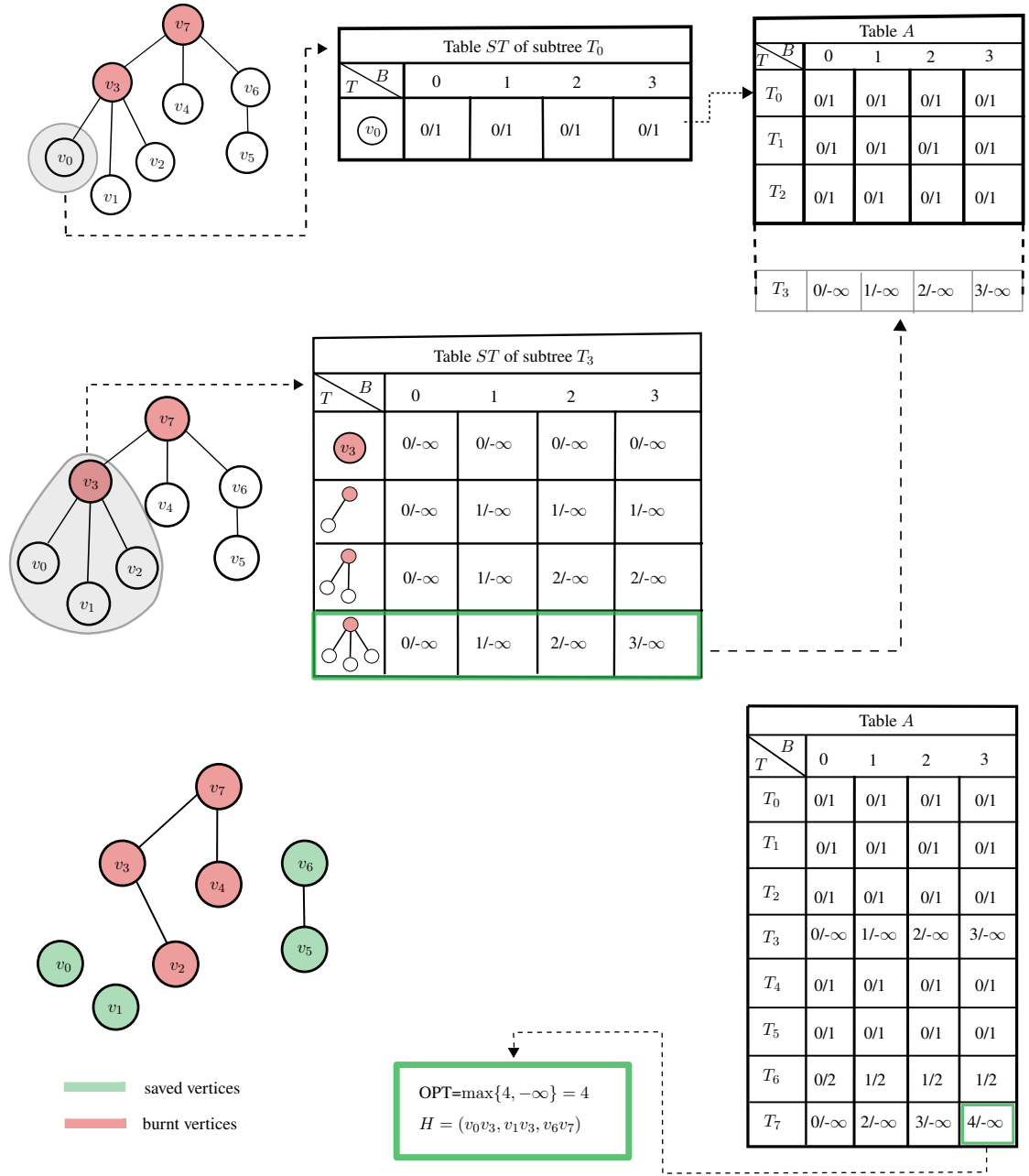}
\caption{Computation of 
\mar{an} optimal {\em cut system} $H$ 
\mar{for the} tree $T$ 
\mar{with} $\pi_i(3)=\pi_i(7)=1$ and a budget $B$=3}
\label{fig:treeAlgov2}
\end{center}
\end{figure}

\clearpage

Figure~\ref{fig:treeAlgov2} shows an example of a tree 
\mar{with} eight vertices, two of them burning (vertices 3 and 7). For simplicity, the tables do not include the related cut systems, but only the number of saved vertices. The top 
\mar{table} depicts the solutions computed for $T_0$, for each possible budget, as stored in Table \pier{ST}. 
The values 0 and 1 in ``0/1'' \mar{correspond} to the number of vertices saved when vertex $v_0$ 
burns and does not burn, respectively. Note that 
these values are not affected by the budget because they refer to a leaf of $T$ (and consequently, the related subtree has no edge). 
These values fill the first row of Table A. The center of Figure~\ref{fig:treeAlgov2} shows Table \pier{ST} for the subtree $T_3$: all possible subtrees to inspect are in the first column, the solutions -- for all budget values -- are in the subsequent columns. 
When 
Table \pier{ST} is complete, its last row 
represents an optimal solution for \pier{$T_i^k$, i.e., the} whole subtree $T_i$, \mar{for the possible states of $v_i$}, and then this row is copied into Table A at row $i$. Finally, the bottom of Figure \ref{fig:treeAlgov2} 
\mar{corresponds to} the Table A, which shows that the optimal value is the maximum between 4 and $-\infty$. The value $4$ then states that four vertices can be saved 
\mar{by} applying the associated cut system (i.e., that cuts the edges $v_3v_0$, $v_3v_1$, and $v_7v_6$).

\subsection{Algorithm correctness and computational complexity}
\label{ssec:algo-trees}

\begin{thm}\label{Theo-tree}
Algorithm~\ref{treecut} correctly computes an optimal solution for an instance    $I=(T,\mathds{1},\pi_i,\mathds{1},\\ \mathds{1},B)$ of \MSFS\ where $T=(V,E)$ is a tree and $\pi_i(v) \in\{1,0\}$.  The computational time is $O(|V| \cdot B^2)$. 
\end{thm}
\begin{proof}
\mar{We prove the correctness by induction on the rank in post order.}
For each leaf 
$v_i$ of $T$, Algorithm~\ref{treecut} correctly computes the values of $A_i$ by calling Algorithm~\ref{tableB}, where the only instructions executed are from line~\ref{algo2:leafstart} to line~\ref{algo2:leafend}, since $v_i$ has no children. \mar{This proves the base step but also, later in the algorithm, the correctness of the values $A_{i,b}$ for any leaf $v_i$.}

Now, given a vertex $v_i$ that is not a leaf (\mar{in particular $i>0$}, we assume \mar{by induction hypothesis} that an optimal solution is computed for each subtree $T_{i_1}, \ldots, T_{i_k}$ rooted at vertices $v_{i_1}, \ldots, v_{i_k}$, children of $v_i$ and for each budget $ b=0,\ldots,B$. These values are stored in $A_{i_j}$, for each $j=1,\ldots,k$ (in particular in each row $i_j$ of Table A). 
\mar{We prove by induction on $j=0, \ldots, k$ that the values $ST_{j,b,f^+}, b=0, \ldots, k$ are correct.}


For $T_i^0$, that is the subtree consisting of vertex $v_i$ only, an optimal solution for each possible budget $b=0,\ldots, B$ is computed  from line~\ref{algo2:leafstart} to line~\ref{algo2:leafend}, and stored in $ST_{0,b}$. Now, assume that $ST_{j-1,b}=(f^+,f^-,H^+,H^-)$ is correctly computed for the subtree $T^{j-1}_i$, for each possible $b$ \mar{and $j\in \{1, \ldots, k\}$. The discussion in Section~\ref{subsec:algo_descript} justifies that the line $j$ of Table ST is then properly filled-in. By induction, it shows that the line $k$ of Table ST will be properly filled-in by Algorithm~\ref{tableB}. Since $T_i^{k}$ is the subtree $T_i$, it completes the proof of the induction step that Algorithm~\ref{treecut} correctly fills Table A in. This completes the proof of the correctness. } 

Regarding the computational complexity, the execution time of Algorithm~\ref{tableB} is dominated by the instruction at line \ref{algo2:bbstart} that requires $O(B)$ time. Since it is repeated $k(B+1)$ times,  Algorithm~\ref{tableB} requires $O(k \cdot B^2)$ time, where $k$ is the number of children for the node into consideration. Algorithm~\ref{tableB} is called for each node of the tree by Algorithm~\ref{treecut}, then the overall computational complexity of Algorithm~\ref{treecut} is \mar{$O(|E| \cdot B^2)$ and since $T$ is a tree, $O(|E|)=O(|V|)$, the complexity is $O(|V| \cdot B^2)$, which completes the proof.} 
\end{proof}

\mar{We can assume $B\leq n-2$ since for larger budget we can cut all edges. So, the complexity of Algorithm \ref{treecut} is dominated by $O(n^3)$.}

\mar{For simplicity of the presentation we have described Algorithm~\ref{tableB} with unitary edge costs and vertex values. However, it can easily} 
be generalized 
\mar{to take} into account the vertices' weights and the edges' costs as follows. 

\mar{Note first that Proposition~\ref{prop:value-reduction} could be applied directly to take into account polynomially bounded vertex values since the class of trees is stable by subdivision of edges. However, a direct generalization of Algorithm~\ref{tableB} does not require the polynomially bounded condition for vertex values. In addition, Proposition~\ref{prop:cost-reduction} cannot be applied since the related transformation does not preserve the class of trees.}

In case of generic vertices' weights, Algorithm \ref{tableB} has to calculate the sum of the weights of the vertices that do not burn, instead of the number of saved vertices, for each subtree. The only required change is in line~\ref{algo2:initnofire}, where $ST_{0,b}$ should be initialized to $(0,\varphi(v_i),\emptyset,\emptyset)$.

With respect to the edges' costs, the Algorithms can be generalized if costs are integer. To this aim, Algorithm~\ref{tableB} should be changed when a cut is needed on the edge $e$, between $v_i$ and the root of $T_{i_{j}}$. If the cost is $c=\kappa(e)$, then the Algorithm should calculate the optimal solutions for a remaining budget $b-c$ (instead of $b-1$), if $b \geq c$. For instance, the condition in line~\ref{algo2:bnbudget} should be changed to $(\pi_i(v_{i_j})=0) \wedge (b \geq c)$, and line~\ref{algo2:bnstart} to $z' \gets \underset{x \in \{0,...,b-c\}}{\arg \max} \{ ST_{j-1,x,f^+} + A_{i_j,b-c-x,f^-} \}$. Similar changes should be applied to lines~\ref{algo2:bnend}, \ref{algo2:bnaddcutsystem}, \ref{algo2:nbbudget}--\ref{algo2:nbend}.

\mar{Algorithm~\ref{treecut} remains unchanged and the overall complexity  $O(|V| \cdot B^2)$ is the same but, this time, we can only assume $B \gab{<} \sum \kappa(e)$ for integral values of edge costs. So, the overall process is polynomial only for polynomially bounded integral edge costs and pseudo polynomial else. }

These considerations lead to the following theorem.

\begin{thm}\label{Theo-tree2}
There exists an algorithm that correctly computes an optimal solution for an instance $I=(T=(V,E),\mathds{1},\pi_i,\kappa,\varphi,B)$ of \MSFS\ where $T$ is a tree, \mar{$\kappa(e) \in \mathbb{N}$, $\forall e\in E$} and $\pi_i(v) \in\{1,0\}$. Its computational complexity is polynomial in $|V|$ if $B= O(Poly(|V|))$ or $\kappa(e) = O(Poly(|V|))$. 
\end{thm}

\section{Conclusions}
\label{sec:conclusions}

In this work, we study the computational complexity of \MFS, introduced in~\cite{safetyscience2021} and the restricted version \MSFS. Both problems are motivated by a wild fire management context. We focus on specific instances, in particular low degree planar graphs, since they make sense for the application under consideration. We  prove the  NP-completeness of \MSFS in planar bipartite graphs using  a new restricted version of Planar Max 2-SAT. Using two successive reductions, we even show that it is still NP-complete in bipartite planar graphs of maximum degree~4 and unitary vertex values and edge costs. On the other hand, we proved that \MSFS is polynomial on trees with polynomially bounded edge costs and binary ignition probabilities.

Our investigations so far still leave a significant gap between the best hardness result and this polynomial case. This motivates some open questions. The main open cases we want to outline are as follows:

\begin{quote}
    \begin{itemize}
        \item  Determine the complexity of \MSFS in subgrids or even in grids, for unitary vertex values and edge costs. \item  Determine the complexity of \MSFS in planar bipartite graphs for unitary vertex values and edge costs and polynomially bounded rational ignition probabilities. 
        \item Determine the complexity of \MSFS for binary ignition probabilities. This case is purely combinatorial in the sense that  the objective value is directly expressed as the  number of vertices  achievable from burning vertices in the graph obtained after removing a cut system.
    \end{itemize}
\end{quote}

The hardness results we obtained motivate the study of the approximation properties of the problem and  identifying new classes of instances solvable in polynomial time. All together these results would help better understanding the problem. In~\cite{safetyscience2021}, we proposed and implemented a heuristic for a generalisation of \MFS. Approximation results and new polynomial cases would give keys to improve  this heuristic for a practical use. 

Finally, an interesting generalisation of \MFS is to take into account the spread of fire by embers. From a graph theory's point of view, it corresponds to adding some edges that cannot be part of cut system (ember edges) and the graph obtained by removing all these ember edges is planar.



\bibliographystyle{abbrv}
\bibliography{references,references_a}

\end{document}

%% file: drawing_v2_latex.pdf_tex
\begingroup%
  \makeatletter%
  \providecommand\color[2][]{%
    \errmessage{(Inkscape) Color is used for the text in Inkscape, but the package 'color.sty' is not loaded}%
    \renewcommand\color[2][]{}%
  }%
  \providecommand\transparent[1]{%
    \errmessage{(Inkscape) Transparency is used (non-zero) for the text in Inkscape, but the package 'transparent.sty' is not loaded}%
    \renewcommand\transparent[1]{}%
  }%
  \providecommand\rotatebox[2]{#2}%
  \newcommand*\fsize{\dimexpr\f@size pt\relax}%
  \newcommand*\lineheight[1]{\fontsize{\fsize}{#1\fsize}\selectfont}%
  \ifx\svgwidth\undefined%
    \setlength{\unitlength}{665.25227224bp}%
    \ifx\svgscale\undefined%
      \relax%
    \else%
      \setlength{\unitlength}{\unitlength * \real{\svgscale}}%
    \fi%
  \else%
    \setlength{\unitlength}{\svgwidth}%
  \fi%
  \global\let\svgwidth\undefined%
  \global\let\svgscale\undefined%
  \makeatother%
  \begin{picture}(1,1.16609364)%
    \lineheight{1}%
    \setlength\tabcolsep{0pt}%
    \put(0.41796128,1.10656028){\color[rgb]{0,0,0}\makebox(0,0)[lt]{\lineheight{1.25}\smash{\begin{tabular}[t]{l}Table $ST$ of subtree $T_0$\end{tabular}}}}%
    \put(0.3250816,1.0687158){\color[rgb]{0,0,0}\makebox(0,0)[lt]{\lineheight{1.25}\smash{\begin{tabular}[t]{l}$T$\end{tabular}}}}%
    \put(0.39614376,1.02780146){\color[rgb]{0,0,0}\makebox(0,0)[lt]{\lineheight{1.25}\smash{\begin{tabular}[t]{l}0/1 \end{tabular}}}}%
    \put(0,0){\includegraphics[width=\unitlength,page=1]{drawing_v2_latex.pdf}}%
    \put(0.35302418,1.07897928){\color[rgb]{0,0,0}\makebox(0,0)[lt]{\lineheight{1.25}\smash{\begin{tabular}[t]{l}$B$\end{tabular}}}}%
    \put(0,0){\includegraphics[width=\unitlength,page=2]{drawing_v2_latex.pdf}}%
    \put(0.31865701,1.1292325){\color[rgb]{0,0,0}\makebox(0,0)[lt]{\begin{minipage}{0.3882624\unitlength}\raggedright \end{minipage}}}%
    \put(0.34373473,1.02880819){\color[rgb]{0,0,0}\makebox(0,0)[lt]{\lineheight{1.25}\smash{\begin{tabular}[t]{l}$v_0$\end{tabular}}}}%
    \put(0.46378715,1.02780146){\color[rgb]{0,0,0}\makebox(0,0)[lt]{\lineheight{1.25}\smash{\begin{tabular}[t]{l}0/1 \end{tabular}}}}%
    \put(0.52917575,1.02780146){\color[rgb]{0,0,0}\makebox(0,0)[lt]{\lineheight{1.25}\smash{\begin{tabular}[t]{l}0/1 \end{tabular}}}}%
    \put(0.60358347,1.02780146){\color[rgb]{0,0,0}\makebox(0,0)[lt]{\lineheight{1.25}\smash{\begin{tabular}[t]{l}0/1 \end{tabular}}}}%
    \put(0.77085588,1.05530438){\color[rgb]{0,0,0}\makebox(0,0)[lt]{\lineheight{1.25}\smash{\begin{tabular}[t]{l}0/1 \end{tabular}}}}%
    \put(0.82046147,1.05530438){\color[rgb]{0,0,0}\makebox(0,0)[lt]{\lineheight{1.25}\smash{\begin{tabular}[t]{l}0/1 \end{tabular}}}}%
    \put(0.87232186,1.05530438){\color[rgb]{0,0,0}\makebox(0,0)[lt]{\lineheight{1.25}\smash{\begin{tabular}[t]{l}0/1 \end{tabular}}}}%
    \put(0.92418225,1.05530438){\color[rgb]{0,0,0}\makebox(0,0)[lt]{\lineheight{1.25}\smash{\begin{tabular}[t]{l}0/1 \end{tabular}}}}%
    \put(0.92418225,1.01020879){\color[rgb]{0,0,0}\makebox(0,0)[lt]{\lineheight{1.25}\smash{\begin{tabular}[t]{l}0/1 \end{tabular}}}}%
    \put(0.92418225,0.96285843){\color[rgb]{0,0,0}\makebox(0,0)[lt]{\lineheight{1.25}\smash{\begin{tabular}[t]{l}0/1 \end{tabular}}}}%
    \put(0.87232181,0.96285843){\color[rgb]{0,0,0}\makebox(0,0)[lt]{\lineheight{1.25}\smash{\begin{tabular}[t]{l}0/1 \end{tabular}}}}%
    \put(0.81820658,0.96285843){\color[rgb]{0,0,0}\makebox(0,0)[lt]{\lineheight{1.25}\smash{\begin{tabular}[t]{l}0/1 \end{tabular}}}}%
    \put(0.77085574,0.96285843){\color[rgb]{0,0,0}\makebox(0,0)[lt]{\lineheight{1.25}\smash{\begin{tabular}[t]{l}0/1 \end{tabular}}}}%
    \put(0.87006701,1.01020879){\color[rgb]{0,0,0}\makebox(0,0)[lt]{\lineheight{1.25}\smash{\begin{tabular}[t]{l}0/1 \end{tabular}}}}%
    \put(0.82046138,1.01020879){\color[rgb]{0,0,0}\makebox(0,0)[lt]{\lineheight{1.25}\smash{\begin{tabular}[t]{l}0/1 \end{tabular}}}}%
    \put(0.77311055,1.01020879){\color[rgb]{0,0,0}\makebox(0,0)[lt]{\lineheight{1.25}\smash{\begin{tabular}[t]{l}0/1 \end{tabular}}}}%
    \put(0.77085588,0.38885653){\color[rgb]{0,0,0}\makebox(0,0)[lt]{\lineheight{1.25}\smash{\begin{tabular}[t]{l}0/1 \end{tabular}}}}%
    \put(0.82046147,0.38885653){\color[rgb]{0,0,0}\makebox(0,0)[lt]{\lineheight{1.25}\smash{\begin{tabular}[t]{l}0/1 \end{tabular}}}}%
    \put(0.87232186,0.38885653){\color[rgb]{0,0,0}\makebox(0,0)[lt]{\lineheight{1.25}\smash{\begin{tabular}[t]{l}0/1 \end{tabular}}}}%
    \put(0.92418225,0.38885653){\color[rgb]{0,0,0}\makebox(0,0)[lt]{\lineheight{1.25}\smash{\begin{tabular}[t]{l}0/1 \end{tabular}}}}%
    \put(0.92418225,0.34376094){\color[rgb]{0,0,0}\makebox(0,0)[lt]{\lineheight{1.25}\smash{\begin{tabular}[t]{l}0/1 \end{tabular}}}}%
    \put(0.92418225,0.29641058){\color[rgb]{0,0,0}\makebox(0,0)[lt]{\lineheight{1.25}\smash{\begin{tabular}[t]{l}0/1 \end{tabular}}}}%
    \put(0.87232181,0.29641058){\color[rgb]{0,0,0}\makebox(0,0)[lt]{\lineheight{1.25}\smash{\begin{tabular}[t]{l}0/1 \end{tabular}}}}%
    \put(0.81820658,0.29641058){\color[rgb]{0,0,0}\makebox(0,0)[lt]{\lineheight{1.25}\smash{\begin{tabular}[t]{l}0/1 \end{tabular}}}}%
    \put(0.77085574,0.29641058){\color[rgb]{0,0,0}\makebox(0,0)[lt]{\lineheight{1.25}\smash{\begin{tabular}[t]{l}0/1 \end{tabular}}}}%
    \put(0.87006701,0.34376094){\color[rgb]{0,0,0}\makebox(0,0)[lt]{\lineheight{1.25}\smash{\begin{tabular}[t]{l}0/1 \end{tabular}}}}%
    \put(0.82046138,0.34376094){\color[rgb]{0,0,0}\makebox(0,0)[lt]{\lineheight{1.25}\smash{\begin{tabular}[t]{l}0/1 \end{tabular}}}}%
    \put(0.77311055,0.34376094){\color[rgb]{0,0,0}\makebox(0,0)[lt]{\lineheight{1.25}\smash{\begin{tabular}[t]{l}0/1 \end{tabular}}}}%
    \put(0.92418225,0.21298359){\color[rgb]{0,0,0}\makebox(0,0)[lt]{\lineheight{1.25}\smash{\begin{tabular}[t]{l}0/1 \end{tabular}}}}%
    \put(0.87232181,0.21298359){\color[rgb]{0,0,0}\makebox(0,0)[lt]{\lineheight{1.25}\smash{\begin{tabular}[t]{l}0/1 \end{tabular}}}}%
    \put(0.81820658,0.21298359){\color[rgb]{0,0,0}\makebox(0,0)[lt]{\lineheight{1.25}\smash{\begin{tabular}[t]{l}0/1 \end{tabular}}}}%
    \put(0.77085574,0.21298359){\color[rgb]{0,0,0}\makebox(0,0)[lt]{\lineheight{1.25}\smash{\begin{tabular}[t]{l}0/1 \end{tabular}}}}%
    \put(0.92418225,0.16788793){\color[rgb]{0,0,0}\makebox(0,0)[lt]{\lineheight{1.25}\smash{\begin{tabular}[t]{l}0/1 \end{tabular}}}}%
    \put(0.87232181,0.16788793){\color[rgb]{0,0,0}\makebox(0,0)[lt]{\lineheight{1.25}\smash{\begin{tabular}[t]{l}0/1 \end{tabular}}}}%
    \put(0.81820658,0.16788793){\color[rgb]{0,0,0}\makebox(0,0)[lt]{\lineheight{1.25}\smash{\begin{tabular}[t]{l}0/1 \end{tabular}}}}%
    \put(0.77085574,0.16788793){\color[rgb]{0,0,0}\makebox(0,0)[lt]{\lineheight{1.25}\smash{\begin{tabular}[t]{l}0/1 \end{tabular}}}}%
    \put(0.92418225,0.12279227){\color[rgb]{0,0,0}\makebox(0,0)[lt]{\lineheight{1.25}\smash{\begin{tabular}[t]{l}1/2 \end{tabular}}}}%
    \put(0.87232181,0.12279227){\color[rgb]{0,0,0}\makebox(0,0)[lt]{\lineheight{1.25}\smash{\begin{tabular}[t]{l}1/2 \end{tabular}}}}%
    \put(0.81820658,0.12279227){\color[rgb]{0,0,0}\makebox(0,0)[lt]{\lineheight{1.25}\smash{\begin{tabular}[t]{l}1/2 \end{tabular}}}}%
    \put(0.77085574,0.12279227){\color[rgb]{0,0,0}\makebox(0,0)[lt]{\lineheight{1.25}\smash{\begin{tabular}[t]{l}0/2 \end{tabular}}}}%
    \put(0,0){\includegraphics[width=\unitlength,page=3]{drawing_v2_latex.pdf}}%
    \put(0.05069942,1.00184016){\color[rgb]{0,0,0}\transparent{0.98000002}\makebox(0,0)[lt]{\lineheight{1.25}\smash{\begin{tabular}[t]{l}$v_0$\end{tabular}}}}%
    \put(0.09449088,0.96504655){\color[rgb]{0,0,0}\transparent{0.98000002}\makebox(0,0)[lt]{\lineheight{1.25}\smash{\begin{tabular}[t]{l}$v_1$\end{tabular}}}}%
    \put(0.13820564,0.99838131){\color[rgb]{0,0,0}\transparent{0.98000002}\makebox(0,0)[lt]{\lineheight{1.25}\smash{\begin{tabular}[t]{l}$v_2$\end{tabular}}}}%
    \put(0.09580341,1.07221373){\color[rgb]{0,0,0}\transparent{0.98000002}\makebox(0,0)[lt]{\lineheight{1.25}\smash{\begin{tabular}[t]{l}$v_3$\end{tabular}}}}%
    \put(0.16192922,1.05126026){\color[rgb]{0,0,0}\transparent{0.98000002}\makebox(0,0)[lt]{\lineheight{1.25}\smash{\begin{tabular}[t]{l}$v_4$\end{tabular}}}}%
    \put(0.16254128,1.12971902){\color[rgb]{0,0,0}\transparent{0.98000002}\makebox(0,0)[lt]{\lineheight{1.25}\smash{\begin{tabular}[t]{l}$v_7$\end{tabular}}}}%
    \put(0.2283369,1.06856974){\color[rgb]{0,0,0}\transparent{0.98000002}\makebox(0,0)[lt]{\lineheight{1.25}\smash{\begin{tabular}[t]{l}$v_6$\end{tabular}}}}%
    \put(0,0){\includegraphics[width=\unitlength,page=4]{drawing_v2_latex.pdf}}%
    \put(0.22392403,1.0050289){\color[rgb]{0,0,0}\transparent{0.98000002}\makebox(0,0)[lt]{\lineheight{1.25}\smash{\begin{tabular}[t]{l}$v_5$\end{tabular}}}}%
    \put(0,0){\includegraphics[width=\unitlength,page=5]{drawing_v2_latex.pdf}}%
    \put(0.02603875,0.61525548){\color[rgb]{0,0,0}\makebox(0,0)[lt]{\lineheight{1.25}\smash{\begin{tabular}[t]{l}$v_0$\end{tabular}}}}%
    \put(0.07169814,0.57636044){\color[rgb]{0,0,0}\makebox(0,0)[lt]{\lineheight{1.25}\smash{\begin{tabular}[t]{l}$v_1$\end{tabular}}}}%
    \put(0.11707566,0.61046405){\color[rgb]{0,0,0}\makebox(0,0)[lt]{\lineheight{1.25}\smash{\begin{tabular}[t]{l}$v_2$\end{tabular}}}}%
    \put(0.07508032,0.69360921){\color[rgb]{0,0,0}\makebox(0,0)[lt]{\lineheight{1.25}\smash{\begin{tabular}[t]{l}$v_3$\end{tabular}}}}%
    \put(0.146106,0.67134322){\color[rgb]{0,0,0}\makebox(0,0)[lt]{\lineheight{1.25}\smash{\begin{tabular}[t]{l}$v_4$\end{tabular}}}}%
    \put(0.14751525,0.75336098){\color[rgb]{0,0,0}\makebox(0,0)[lt]{\lineheight{1.25}\smash{\begin{tabular}[t]{l}$v_7$\end{tabular}}}}%
    \put(0.21684985,0.68853594){\color[rgb]{0,0,0}\makebox(0,0)[lt]{\lineheight{1.25}\smash{\begin{tabular}[t]{l}$v_6$\end{tabular}}}}%
    \put(0.21515876,0.61976503){\color[rgb]{0,0,0}\makebox(0,0)[lt]{\lineheight{1.25}\smash{\begin{tabular}[t]{l}$v_5$\end{tabular}}}}%
    \put(0,0){\includegraphics[width=\unitlength,page=6]{drawing_v2_latex.pdf}}%
    \put(0.19565979,0.40201577){\color[rgb]{0,0,0}\makebox(0,0)[lt]{\lineheight{1.25}\smash{\begin{tabular}[t]{l}$v_7$\end{tabular}}}}%
    \put(0.27957752,0.32982391){\color[rgb]{0,0,0}\makebox(0,0)[lt]{\lineheight{1.25}\smash{\begin{tabular}[t]{l}$v_6$\end{tabular}}}}%
    \put(0.27884115,0.24928645){\color[rgb]{0,0,0}\makebox(0,0)[lt]{\lineheight{1.25}\smash{\begin{tabular}[t]{l}$v_5$\end{tabular}}}}%
    \put(0.19574786,0.30985004){\color[rgb]{0,0,0}\makebox(0,0)[lt]{\lineheight{1.25}\smash{\begin{tabular}[t]{l}$v_4$\end{tabular}}}}%
    \put(0.10487792,0.32902384){\color[rgb]{0,0,0}\makebox(0,0)[lt]{\lineheight{1.25}\smash{\begin{tabular}[t]{l}$v_3$\end{tabular}}}}%
    \put(0.04922588,0.23344324){\color[rgb]{0,0,0}\makebox(0,0)[lt]{\lineheight{1.25}\smash{\begin{tabular}[t]{l}$v_0$\end{tabular}}}}%
    \put(0.09846792,0.19312539){\color[rgb]{0,0,0}\makebox(0,0)[lt]{\lineheight{1.25}\smash{\begin{tabular}[t]{l}$v_1$\end{tabular}}}}%
    \put(0.15735479,0.23312953){\color[rgb]{0,0,0}\makebox(0,0)[lt]{\lineheight{1.25}\smash{\begin{tabular}[t]{l}$v_2$\end{tabular}}}}%
    \put(0,0){\includegraphics[width=\unitlength,page=7]{drawing_v2_latex.pdf}}%
    \put(0.13333087,0.11478649){\color[rgb]{0,0,0}\makebox(0,0)[lt]{\lineheight{1.25}\smash{\begin{tabular}[t]{l}saved vertices\end{tabular}}}}%
    \put(0,0){\includegraphics[width=\unitlength,page=8]{drawing_v2_latex.pdf}}%
    \put(0.13174474,0.08101103){\color[rgb]{0,0,0}\makebox(0,0)[lt]{\lineheight{1.25}\smash{\begin{tabular}[t]{l}burnt vertices\end{tabular}}}}%
    \put(0,0){\includegraphics[width=\unitlength,page=9]{drawing_v2_latex.pdf}}%
    \put(0.81402815,1.12140157){\color[rgb]{0,0,0}\makebox(0,0)[lt]{\lineheight{1.25}\smash{\begin{tabular}[t]{l}Table $A$\end{tabular}}}}%
    \put(0.71592429,1.09360041){\color[rgb]{0,0,0}\makebox(0,0)[lt]{\lineheight{1.25}\smash{\begin{tabular}[t]{l}$T$\end{tabular}}}}%
    \put(0.77007468,1.06041667){\color[rgb]{0,0,0}\makebox(0,0)[lt]{\lineheight{1.25}\smash{\begin{tabular}[t]{l} \end{tabular}}}}%
    \put(0.76948794,1.01644341){\color[rgb]{0,0,0}\makebox(0,0)[lt]{\lineheight{1.25}\smash{\begin{tabular}[t]{l} \end{tabular}}}}%
    \put(0.76831446,0.96895228){\color[rgb]{0,0,0}\makebox(0,0)[lt]{\lineheight{1.25}\smash{\begin{tabular}[t]{l} \end{tabular}}}}%
    \put(0,0){\includegraphics[width=\unitlength,page=10]{drawing_v2_latex.pdf}}%
    \put(0.72795807,0.96935378){\color[rgb]{0,0,0}\makebox(0,0)[lt]{\lineheight{1.25}\smash{\begin{tabular}[t]{l}$T_2$\end{tabular}}}}%
    \put(0.72642004,1.01885736){\color[rgb]{0,0,0}\makebox(0,0)[lt]{\lineheight{1.25}\smash{\begin{tabular}[t]{l}$T_1$\end{tabular}}}}%
    \put(0.72453621,1.06258526){\color[rgb]{0,0,0}\makebox(0,0)[lt]{\lineheight{1.25}\smash{\begin{tabular}[t]{l}$T_0$\end{tabular}}}}%
    \put(0.7392224,1.10300282){\color[rgb]{0,0,0}\makebox(0,0)[lt]{\lineheight{1.25}\smash{\begin{tabular}[t]{l}$B$\end{tabular}}}}%
    \put(0,0){\includegraphics[width=\unitlength,page=11]{drawing_v2_latex.pdf}}%
    \put(0.40870696,0.81217685){\color[rgb]{0,0,0}\makebox(0,0)[lt]{\lineheight{1.25}\smash{\begin{tabular}[t]{l}Table $ST$ of subtree $T_3$\end{tabular}}}}%
    \put(0.30072339,0.77199057){\color[rgb]{0,0,0}\makebox(0,0)[lt]{\lineheight{1.25}\smash{\begin{tabular}[t]{l}$T$ \end{tabular}}}}%
    \put(0,0){\includegraphics[width=\unitlength,page=12]{drawing_v2_latex.pdf}}%
    \put(0.33902806,0.80106232){\color[rgb]{0,0,0}\makebox(0,0)[lt]{\begin{minipage}{0.38643172\unitlength}\raggedright \end{minipage}}}%
    \put(0,0){\includegraphics[width=\unitlength,page=13]{drawing_v2_latex.pdf}}%
    \put(0.73081141,0.87631014){\color[rgb]{0,0,0}\makebox(0,0)[lt]{\lineheight{1.25}\smash{\begin{tabular}[t]{l}$T_3$\end{tabular}}}}%
    \put(0,0){\includegraphics[width=\unitlength,page=14]{drawing_v2_latex.pdf}}%
    \put(0.31997739,0.72585826){\color[rgb]{0,0,0}\makebox(0,0)[lt]{\lineheight{1.25}\smash{\begin{tabular}[t]{l}$v_3$\end{tabular}}}}%
    \put(0,0){\includegraphics[width=\unitlength,page=15]{drawing_v2_latex.pdf}}%
    \put(0.77586949,0.4281335){\color[rgb]{0,0,0}\makebox(0,0)[lt]{\lineheight{1.25}\smash{\begin{tabular}[t]{l}0\end{tabular}}}}%
    \put(0.82293655,0.42750143){\color[rgb]{0,0,0}\makebox(0,0)[lt]{\lineheight{1.25}\smash{\begin{tabular}[t]{l}1\end{tabular}}}}%
    \put(0.87343804,0.42759191){\color[rgb]{0,0,0}\makebox(0,0)[lt]{\lineheight{1.25}\smash{\begin{tabular}[t]{l}2\end{tabular}}}}%
    \put(0.92912097,0.4282323){\color[rgb]{0,0,0}\makebox(0,0)[lt]{\lineheight{1.25}\smash{\begin{tabular}[t]{l}3\end{tabular}}}}%
    \put(0.76793706,0.08335124){\color[rgb]{0,0,0}\makebox(0,0)[lt]{\lineheight{1.25}\smash{\begin{tabular}[t]{l}0/-$\infty$\end{tabular}}}}%
    \put(0.81528688,0.08335124){\color[rgb]{0,0,0}\makebox(0,0)[lt]{\lineheight{1.25}\smash{\begin{tabular}[t]{l}2/-$\infty$\end{tabular}}}}%
    \put(0.86714621,0.08335124){\color[rgb]{0,0,0}\makebox(0,0)[lt]{\lineheight{1.25}\smash{\begin{tabular}[t]{l}3/-$\infty$\end{tabular}}}}%
    \put(0.91675078,0.08335124){\color[rgb]{0,0,0}\makebox(0,0)[lt]{\lineheight{1.25}\smash{\begin{tabular}[t]{l}4/-$\infty$\end{tabular}}}}%
    \put(0.76793706,0.26147921){\color[rgb]{0,0,0}\makebox(0,0)[lt]{\lineheight{1.25}\smash{\begin{tabular}[t]{l}0/-$\infty$\end{tabular}}}}%
    \put(0.81528688,0.26147921){\color[rgb]{0,0,0}\makebox(0,0)[lt]{\lineheight{1.25}\smash{\begin{tabular}[t]{l}1/-$\infty$\end{tabular}}}}%
    \put(0.86489145,0.26147921){\color[rgb]{0,0,0}\makebox(0,0)[lt]{\lineheight{1.25}\smash{\begin{tabular}[t]{l}2/-$\infty$\end{tabular}}}}%
    \put(0.91900554,0.26147921){\color[rgb]{0,0,0}\makebox(0,0)[lt]{\lineheight{1.25}\smash{\begin{tabular}[t]{l}3/-$\infty$\end{tabular}}}}%
    \put(0.81371265,0.46471303){\color[rgb]{0,0,0}\makebox(0,0)[lt]{\lineheight{1.25}\smash{\begin{tabular}[t]{l}Table $A$\end{tabular}}}}%
    \put(0.7147027,0.4232089){\color[rgb]{0,0,0}\makebox(0,0)[lt]{\lineheight{1.25}\smash{\begin{tabular}[t]{l}$T$\end{tabular}}}}%
    \put(0,0){\includegraphics[width=\unitlength,page=16]{drawing_v2_latex.pdf}}%
    \put(0.73885318,0.44078774){\color[rgb]{0,0,0}\makebox(0,0)[lt]{\lineheight{1.25}\smash{\begin{tabular}[t]{l}$B$\end{tabular}}}}%
    \put(0.72217911,0.39085503){\color[rgb]{0,0,0}\makebox(0,0)[lt]{\lineheight{1.25}\smash{\begin{tabular}[t]{l}$T_0$\end{tabular}}}}%
    \put(0.72195231,0.34936244){\color[rgb]{0,0,0}\makebox(0,0)[lt]{\lineheight{1.25}\smash{\begin{tabular}[t]{l}$T_1$\end{tabular}}}}%
    \put(0.72195231,0.3014483){\color[rgb]{0,0,0}\makebox(0,0)[lt]{\lineheight{1.25}\smash{\begin{tabular}[t]{l}$T_2$\end{tabular}}}}%
    \put(0.72195231,0.2591711){\color[rgb]{0,0,0}\makebox(0,0)[lt]{\lineheight{1.25}\smash{\begin{tabular}[t]{l}$T_3$\end{tabular}}}}%
    \put(0.7230797,0.21520279){\color[rgb]{0,0,0}\makebox(0,0)[lt]{\lineheight{1.25}\smash{\begin{tabular}[t]{l}$T_4$\end{tabular}}}}%
    \put(0.72195231,0.17123455){\color[rgb]{0,0,0}\makebox(0,0)[lt]{\lineheight{1.25}\smash{\begin{tabular}[t]{l}$T_5$\end{tabular}}}}%
    \put(0.72195231,0.12726627){\color[rgb]{0,0,0}\makebox(0,0)[lt]{\lineheight{1.25}\smash{\begin{tabular}[t]{l}$T_6$\end{tabular}}}}%
    \put(0.7230797,0.0838617){\color[rgb]{0,0,0}\makebox(0,0)[lt]{\lineheight{1.25}\smash{\begin{tabular}[t]{l}$T_7$\end{tabular}}}}%
    \put(0.33732275,0.78264909){\color[rgb]{0,0,0}\makebox(0,0)[lt]{\lineheight{1.25}\smash{\begin{tabular}[t]{l}$B$\end{tabular}}}}%
    \put(0,0){\includegraphics[width=\unitlength,page=17]{drawing_v2_latex.pdf}}%
    \put(0.39493725,0.11094169){\color[rgb]{0,0,0}\makebox(0,0)[lt]{\lineheight{1.25}\smash{\begin{tabular}[t]{l}OPT=$\max\{4,-\infty\}=4$\\\end{tabular}}}}%
    \put(0.39482939,0.08143595){\color[rgb]{0,0,0}\makebox(0,0)[lt]{\lineheight{1.25}\smash{\begin{tabular}[t]{l}$H=(v_0v_3, v_1v_3, v_6v_7)$\end{tabular}}}}%
    \put(0,0){\includegraphics[width=\unitlength,page=18]{drawing_v2_latex.pdf}}%
    \put(1.6769684,0.79684621){\color[rgb]{0,0,0}\makebox(0,0)[lt]{\begin{minipage}{0.20789621\unitlength}\raggedright \end{minipage}}}%
    \put(0.40728162,1.07355329){\color[rgb]{0,0,0}\makebox(0,0)[lt]{\lineheight{1.25}\smash{\begin{tabular}[t]{l}0\end{tabular}}}}%
    \put(0.47464172,1.07292123){\color[rgb]{0,0,0}\makebox(0,0)[lt]{\lineheight{1.25}\smash{\begin{tabular}[t]{l}1\end{tabular}}}}%
    \put(0.54092672,1.07301171){\color[rgb]{0,0,0}\makebox(0,0)[lt]{\lineheight{1.25}\smash{\begin{tabular}[t]{l}2\end{tabular}}}}%
    \put(0.61295685,1.0730884){\color[rgb]{0,0,0}\makebox(0,0)[lt]{\lineheight{1.25}\smash{\begin{tabular}[t]{l}3\end{tabular}}}}%
    \put(0.38762905,0.77466147){\color[rgb]{0,0,0}\makebox(0,0)[lt]{\lineheight{1.25}\smash{\begin{tabular}[t]{l}0\end{tabular}}}}%
    \put(0.46215468,0.77396218){\color[rgb]{0,0,0}\makebox(0,0)[lt]{\lineheight{1.25}\smash{\begin{tabular}[t]{l}1\end{tabular}}}}%
    \put(0.53549084,0.77406228){\color[rgb]{0,0,0}\makebox(0,0)[lt]{\lineheight{1.25}\smash{\begin{tabular}[t]{l}2\end{tabular}}}}%
    \put(0.61518329,0.77414713){\color[rgb]{0,0,0}\makebox(0,0)[lt]{\lineheight{1.25}\smash{\begin{tabular}[t]{l}3\end{tabular}}}}%
    \put(0.78044832,1.09779222){\color[rgb]{0,0,0}\makebox(0,0)[lt]{\lineheight{1.25}\smash{\begin{tabular}[t]{l}0\end{tabular}}}}%
    \put(0.82751538,1.09716016){\color[rgb]{0,0,0}\makebox(0,0)[lt]{\lineheight{1.25}\smash{\begin{tabular}[t]{l}1\end{tabular}}}}%
    \put(0.87801688,1.09725064){\color[rgb]{0,0,0}\makebox(0,0)[lt]{\lineheight{1.25}\smash{\begin{tabular}[t]{l}2\end{tabular}}}}%
    \put(0.93369981,1.09789102){\color[rgb]{0,0,0}\makebox(0,0)[lt]{\lineheight{1.25}\smash{\begin{tabular}[t]{l}3\end{tabular}}}}%
    \put(0.77026105,0.87476174){\color[rgb]{0,0,0}\makebox(0,0)[lt]{\lineheight{1.25}\smash{\begin{tabular}[t]{l}0/-$\infty$\end{tabular}}}}%
    \put(0.81761087,0.87476174){\color[rgb]{0,0,0}\makebox(0,0)[lt]{\lineheight{1.25}\smash{\begin{tabular}[t]{l}1/-$\infty$\end{tabular}}}}%
    \put(0.8694702,0.87476174){\color[rgb]{0,0,0}\makebox(0,0)[lt]{\lineheight{1.25}\smash{\begin{tabular}[t]{l}2/-$\infty$\end{tabular}}}}%
    \put(0.92358428,0.87476174){\color[rgb]{0,0,0}\makebox(0,0)[lt]{\lineheight{1.25}\smash{\begin{tabular}[t]{l}3/-$\infty$\end{tabular}}}}%
    \put(0.37913984,0.72887665){\color[rgb]{0,0,0}\makebox(0,0)[lt]{\lineheight{1.25}\smash{\begin{tabular}[t]{l}0/-$\infty$\end{tabular}}}}%
    \put(0.44903719,0.72887665){\color[rgb]{0,0,0}\makebox(0,0)[lt]{\lineheight{1.25}\smash{\begin{tabular}[t]{l}0/-$\infty$\end{tabular}}}}%
    \put(0.5211893,0.72887665){\color[rgb]{0,0,0}\makebox(0,0)[lt]{\lineheight{1.25}\smash{\begin{tabular}[t]{l}0/-$\infty$\end{tabular}}}}%
    \put(0.59559616,0.72887665){\color[rgb]{0,0,0}\makebox(0,0)[lt]{\lineheight{1.25}\smash{\begin{tabular}[t]{l}0/-$\infty$\end{tabular}}}}%
    \put(0.3780691,0.66553553){\color[rgb]{0,0,0}\makebox(0,0)[lt]{\lineheight{1.25}\smash{\begin{tabular}[t]{l}0/-$\infty$\end{tabular}}}}%
    \put(0.44796645,0.66553553){\color[rgb]{0,0,0}\makebox(0,0)[lt]{\lineheight{1.25}\smash{\begin{tabular}[t]{l}1/-$\infty$\end{tabular}}}}%
    \put(0.52011856,0.66553553){\color[rgb]{0,0,0}\makebox(0,0)[lt]{\lineheight{1.25}\smash{\begin{tabular}[t]{l}1/-$\infty$\end{tabular}}}}%
    \put(0.59452542,0.66553553){\color[rgb]{0,0,0}\makebox(0,0)[lt]{\lineheight{1.25}\smash{\begin{tabular}[t]{l}1/-$\infty$\end{tabular}}}}%
    \put(0.37966348,0.60112366){\color[rgb]{0,0,0}\makebox(0,0)[lt]{\lineheight{1.25}\smash{\begin{tabular}[t]{l}0/-$\infty$\end{tabular}}}}%
    \put(0.44956083,0.60112366){\color[rgb]{0,0,0}\makebox(0,0)[lt]{\lineheight{1.25}\smash{\begin{tabular}[t]{l}1/-$\infty$\end{tabular}}}}%
    \put(0.52171294,0.60112366){\color[rgb]{0,0,0}\makebox(0,0)[lt]{\lineheight{1.25}\smash{\begin{tabular}[t]{l}2/-$\infty$\end{tabular}}}}%
    \put(0.5961198,0.60112366){\color[rgb]{0,0,0}\makebox(0,0)[lt]{\lineheight{1.25}\smash{\begin{tabular}[t]{l}2/-$\infty$\end{tabular}}}}%
    \put(0.37804564,0.53893559){\color[rgb]{0,0,0}\makebox(0,0)[lt]{\lineheight{1.25}\smash{\begin{tabular}[t]{l}0/-$\infty$\end{tabular}}}}%
    \put(0.447943,0.53893559){\color[rgb]{0,0,0}\makebox(0,0)[lt]{\lineheight{1.25}\smash{\begin{tabular}[t]{l}1/-$\infty$\end{tabular}}}}%
    \put(0.52009511,0.53893559){\color[rgb]{0,0,0}\makebox(0,0)[lt]{\lineheight{1.25}\smash{\begin{tabular}[t]{l}2/-$\infty$\end{tabular}}}}%
    \put(0.59450197,0.53893559){\color[rgb]{0,0,0}\makebox(0,0)[lt]{\lineheight{1.25}\smash{\begin{tabular}[t]{l}3/-$\infty$\end{tabular}}}}%
  \end{picture}%
\endgroup%